\documentclass[11pt]{amsart}

\usepackage{amssymb, amsfonts, amsmath, amsthm, mathtools}

\usepackage[T1]{fontenc}
\usepackage[sc]{mathpazo}
\usepackage[mathscr]{eucal}
\usepackage[scaled=1]{inconsolata}   

\usepackage{titlesec}
\usepackage{titletoc}      
\titleformat{\section}
  {\normalfont\Large\bfseries}{\S\thesection}{1em}{}
\titleformat{\subsection}
  {\normalfont\large\bfseries}{\S\thesubsection}{1em}{}

\usepackage{hyperref}      

\renewcommand{\emph}{\textbf}

\newcommand{\R}{\mathbb{R}}
\newcommand{\FF}{\mathcal F}
\newcommand{\KK}{\mathcal{K}}

\DeclareMathOperator{\tr}{tr}
\DeclareMathOperator{\supp}{supp}
\DeclareMathOperator{\Rea}{Re}
\DeclareMathOperator{\interior}{int}
\newcommand{\trm}{\operatorname{tr}_m}
\newcommand{\id}{\mathrm{id}}
\newcommand{\Id}{I}
\newcommand{\eps}{\varepsilon}
\renewcommand{\bar}[1]{\overline{#1}}
\newcommand{\ip}[2]{\left\langle #1,#2\right\rangle}   

\newcommand{\Cop}[1]{\mathcal C_{#1}}    
\newcommand{\Fset}{\mathcal{F}}          
\newcommand{\Fint}{\mathcal F^{\circ}}   
\newcommand{\Freg}{F}                    
\newcommand{\Kfun}{\mathscr K}           
\newcommand{\Act}{\mathcal A}            

\theoremstyle{plain}
\newtheorem{theo}{Theorem}[section]
\newtheorem{theorem}[theo]{Theorem}
\newtheorem*{theo*}{Theorem}      
\newtheorem{lemma}[theo]{Lemma}

\newtheorem{proposition}[theo]{Proposition}

\newtheorem{corollary}[theo]{Corollary}

\theoremstyle{definition}

\newtheorem{definition}[theo]{Definition}

\newtheorem{question}{Question}

\theoremstyle{remark}
\newtheorem{remark}[theo]{Remark}
\newtheorem{exa}[theo]{Example}

\newtheorem*{solu*}{Solution}

\title{
Fuglede--Kadison Determinants of Shifted Matrix Semicircles
}

\bigskip
\author{
Vladislav Kargin
}
\thanks{email:
vkargin@binghamton.edu; current address: 4400 Vestal Pkwy East, Department of Mathematics, Binghamton University, Binghamton, 13902-6000, USA}

\begin{document}

\maketitle
\begin{abstract}
We prove a variational formula for the Fuglede--Kadison determinant of a
self-adjoint matrix semicircle
\[
   S_a=a\otimes 1+\sum_j A_j\otimes s_j
\]
with deterministic mean $a$.  For every $z\in\mathbb R$,
\[
   \log\Delta(S_a-z)
   =
   -\frac12+\inf_{b\in\mathcal F}F_z(b),
\]
where $\mathcal F$ is the set of positive matrices $b$ satisfying
$\mathcal C_b-\eta\succeq0$, and $F_z$ is an explicit geodesically convex
functional.  Both sides equal $-\infty$ exactly when $z$ is an atom of
the spectral measure.  The formula was proposed, without proof, in a
recent post of Speicher.  In the centered case $a=0$ it gives a new proof
of the Mai--Speicher capacity formula, and by hermitization it yields a
finite-dimensional variational formula for the logarithmic potential of
the Brown measure of an arbitrary matrix-valued semicircular element.

The proof identifies the determinant with the value of the Dyson action
at its critical point and rewrites this critical-point problem as a
saddle problem in three matrix variables $(b,u,\gamma)$.  For a spectral
parameter $z+i\varepsilon$ in the upper half-plane, eliminating $u$ and
$\gamma$ gives a geodesically convex minimization whose global minimizer
is the positive factor $b_\varepsilon$ in a symmetric unitary
factorization of the matrix Cauchy transform.  The formula on the real
axis follows by letting $\varepsilon\downarrow0$ in the optimal values,
with no compactness assumption on the minimizers.  At points where the
boundary value of the Cauchy transform is nondegenerate, the scales
$b_\varepsilon$ converge and the infimum is attained at their limit
$b_z$; this minimizer lies on the boundary of $\mathcal F$ when $z$ is
in the support of the spectral measure and in its interior when $z$ is
in the resolvent set.
\end{abstract}

\newpage
\contentsmargin{2.55em}
\dottedcontents{section}[3.8em]{}{2.3em}{1pc}
\dottedcontents{subsection}[6.1em]{}{3.2em}{1pc}
\tableofcontents  
\newpage

\section{Introduction and main results}
\label{sec:introduction}

Operator-valued semicircular elements play a central role in free
probability and in the asymptotic analysis of structured random matrices.
We consider self-adjoint matrix semicircles of the form
\[
   S_a
   =
   a\otimes 1
   +
   \sum_{j=1}^r A_j\otimes s_j,
\]
where $a,A_1,\ldots,A_r\in M_m(\mathbb C)$ are Hermitian and
$s_1,\ldots,s_r$ are free standard semicircular variables in a tracial
noncommutative probability space.  The associated covariance map is
\[
   \eta(X)
   =
   \sum_{j=1}^r A_jXA_j .
\]
The operator-valued Cauchy transform $G(w)$ is determined by the
matrix Dyson equation
\begin{equation} \label{eq:dyson}
   G(w)^{-1}+\eta(G(w))=wI-a,
\end{equation}
see, for example, \cite{HeltonRashidiFarSpeicher2007}.

Our object of interest is the Fuglede--Kadison determinant
\cite{FugledeKadison1952},
\[
   \Delta(S_a-z)
   =
   \exp\left\{
      \int_{\mathbb R}\log|t-z|\,d\mu_{S_a}(t)
   \right\},
   \qquad z\in\mathbb C,
\]
where $\mu_{S_a}$ is the scalar spectral measure of $S_a$.  Since $S_a-z$
is a normal operator, this is its Fuglede--Kadison determinant for every
complex $z$.  We are mainly interested in real $z$, where
$\log\Delta(S_a-z)$ is the logarithmic potential of $\mu_{S_a}$ on the
real line; non-real $z$ enter as a regularization.

Two earlier results suggest that this quantity should admit a
finite-dimensional variational description.  On the one hand, Lehner's
formula  \cite{Lehner1999} expresses the spectral edge of a matrix semicircle through an
optimization over positive definite matrices; in our notation its basic
structure is
\[
   \lambda_{\max}(S_a)
   =
   \inf_{b\succ0}
   \lambda_{\max}
   \bigl(
      a+b+\eta(b^{-1})
   \bigr).
\]
On the other hand, Mai and Speicher \cite{MaiSpeicher2024} proved that in the centered case
$a=0$ the Fuglede--Kadison determinant is determined by Gurvits'
capacity of the covariance map:
\[
   \Delta(S_0)
   =
   e^{-1/2}\,
   \operatorname{Cap}(\eta)^{1/(2m)}.
\]
The purpose of the present paper is to give a shifted analogue of the
Mai--Speicher formula and to explain its relation with the Dyson equation.

The question of whether the Mai--Speicher formula admits an extension to
nonzero mean was raised by Speicher in a recent expository
post~\cite{SpeicherBlog2026}. The post also records a candidate formula for
the general case, obtained with the assistance of a large language model
and stated there explicitly as a conjecture, without proof and with the
remark that a formal verification would not by itself explain why such a
formula should hold. The present paper proves that formula and, we hope,
supplies the explanation.

While this paper was being completed, Sheng Yin posted a manuscript
\cite{Yin2026} containing a different, unconstrained variational formula
for the same determinant and, by a different argument, a proof of the
version of \eqref{eq:intro-main-formula} with the infimum taken over
$\operatorname{int}\mathcal F$; see Section~\ref{sec:further-questions}.

Namely, the determinant arises as the critical value of the Dyson action. Its Euler--Lagrange equation is the matrix Dyson equation, so one may compute the determinant by solving for the physical Dyson critical point. The same critical-point problem admits an extended formulation in variables $(b,u,\gamma)$. Maximizing first in $u$ and $\gamma$ produces the functional $F_z(b)$. The advantage of this formulation is geometric: on the feasible set $\mathcal F$, $F_z$ is geodesically convex, and the scale selected by the Dyson solution becomes a global minimizer. Thus the variational formula may be viewed as a convex reduction of the Dyson action.

For $b\succ0$ let
\[
   \mathcal C_b(X):=bXb,
   \qquad
   L_b^+:=\mathcal C_b+\eta, \qquad  L_b^-:=\mathcal C_b-\eta, \qquad a_z := a - z,  
\]
and introduce the functionals
\begin{equation}\label{eq:intro-KK}
   \KK_z(b, u, \gamma)
   :=
   \tr_m\log b + \langle a_z,u\rangle
      -
      \frac12 \left\{
\langle u, L_b^+ u\rangle + \langle \gamma, L_b^-\gamma\rangle
   \right\}, 
\end{equation}
where $\langle X, Y\rangle \equiv \langle X, Y\rangle_2 := \tr_m(X^* Y)$, 
and
\begin{equation}\label{eq:intro-F0}
   F_z(b)
   :=
   \sup_{u=u^*, \gamma \succeq 0}
   \left\{ \KK_z(b, u, \gamma)\right\}.
\end{equation}

Define the feasible set
$
   \FF
   :=
   \left\{
      b\succ0: L_b^- \succeq0
   \right\},
$
where positivity is understood on the Hilbert--Schmidt space
$M_m(\mathbb C)$ and write $\FF^\circ =\interior \FF$. 

Then the formula can be re-written as 
\begin{equation}\label{eq:intro-F}
   F_z(b)
   = \begin{cases}
   \tr_m\log b 
   +
   \sup_{u=u^*}
   \left\{
      \langle a_z,u\rangle
      -
      \frac12\langle u,L_b^+ u\rangle
   \right\}, & b \in \FF, \\
   +\infty & b \notin \FF,
   \end{cases}
\end{equation}
and if $b \in \FF^\circ$, then $L_b^+\succ0$ and
\[
   F_z(b)
   =
   \tr_m\log b
   +
   \frac12
   \left\langle
      a_z,(L_b^+)^{-1}a_z
   \right\rangle.
\]

The formula proposed in~\cite{SpeicherBlog2026}, in our notation, reads
\begin{equation}\label{eq:intro-main-formula}
   \boxed{
   \log\Delta(S_a-z)
   =
   -\frac12
   +
   \inf_{b\in\mathcal F}F_z(b).
   }
\end{equation}

The proof produces a distinguished family of competitors.  For
$\operatorname{Im}w>0$ the matrix $G(w)$ is invertible with
$-\operatorname{Im}G(w)\succ0$, and it admits a symmetric factorization
\begin{equation}\label{eq:intro-pup}
   G(w)
   =
   b(w)^{-1/2}\,U(w)\,b(w)^{-1/2},
   \qquad
   b(w)\succ0,\quad U(w)\in U(m),
\end{equation}
constructed explicitly in Section~\ref{sec:normalizers}.  For real $z$
we write $b_\varepsilon:=b(z+i\varepsilon)$ and call these the
\emph{regularized scales} at $z$.  Our main result is the following.

\begin{theorem}[Variational formula at every real point]
\label{thm:intro-all-points}
For every $z\in\mathbb R$,
\begin{equation}\label{eq:intro-main}
   \inf_{b\in\mathcal F}F_z(b)
   =
   \log\Delta(S_a-z)+\frac12
   \qquad\text{in }[-\infty,\infty).
\end{equation}
Both sides equal $-\infty$ if and only if $z$ is an atom of $\mu_{S_a}$.
The regularized scales form a minimizing sequence:
$F_z(b_\varepsilon)\to\log\Delta(S_a-z)+\frac12$ as
$\varepsilon\downarrow0$.
\end{theorem}

Whether the infimum is attained depends on the boundary behaviour of the
Cauchy transform at $z$.

\begin{definition}[Nondegenerate regular point]
\label{def:nondegenerate-regular}
A point $z\in\mathbb R$ is called \emph{nondegenerate regular} if the
boundary value
$
   G(z+i0)
$
exists, and either
$
   -\operatorname{Im}G(z+i0)\succ0,
$
or $z$ belongs to an open component of
$
   \mathbb R\setminus\supp\mu_{S_a}.
$
\end{definition}

Whenever the boundary value exists it is automatically invertible
(Lemma~\ref{lem:boundary-invertible}).  Nondegenerate regular points
include all points at which the spectral density is positive definite as
a matrix, and all points of the resolvent set.

\begin{theorem}[Attainment at nondegenerate regular points]
\label{thm:intro-attainment}
Let $z$ be a nondegenerate regular point.  Then the regularized scales
converge, $b_z:=\lim_{\varepsilon\downarrow0}b_\varepsilon$ exists, and
the infimum in \eqref{eq:intro-main} is attained at $b_z$.  If $-\operatorname{Im}G(z+i0)\succ0$, then $b_z\in\partial\mathcal F$.
If $z$ lies in the resolvent set, then
$b_z\in\operatorname{int}\mathcal F$; if moreover $z$ lies in an
unbounded component of the resolvent set, then
$
   b_z=|G(z)|^{-1}.
$
\end{theorem}

At points where $-\operatorname{Im}G(z+i0)$ is nonzero but singular,
which can fill the whole support of $\mu_{S_a}$ (Example~\ref{exa:coupled}),
attainment follows from boundedness of $b_\eps$ along a sequence, which we do
not know how to prove in general. 

In the centered case the variational problem depends only on the
covariance map and reduces to Gurvits' capacity.  Recall
\[
   \operatorname{Cap}(\eta)
   :=
   \inf
   \left\{
      \det\eta(X):
      X\succ0,\ \det X=1
   \right\},
\]
and write $\Delta(b):=(\det b)^{1/m}=\exp\bigl(\tr_m\log b\bigr)$ for
$b\succ0$.

\begin{theorem}[Centered problem and capacity]
\label{thm:intro-capacity}
For $a=0$,
\begin{equation}\label{eq:intro-capacity}
   \inf_{b\in\mathcal F}\Delta(b)
   =
   \operatorname{Cap}(\eta)^{1/(2m)}.
\end{equation}
Consequently,
\[
   \log\Delta(S_0)
   =
   -\frac12
   +
   \frac1{2m}\log\operatorname{Cap}(\eta),
\]
which is the Mai--Speicher formula.  Moreover, the infimum in
\eqref{eq:intro-capacity} is attained if and only if $\eta$ is
symmetrically doubly-stochastic scalable.
\end{theorem}

This gives a proof of the Mai--Speicher formula that is independent of
\cite{MaiSpeicher2024}.  Its external inputs are the matrix Dyson equation
\cite{HeltonRashidiFarSpeicher2007}, the Perron--Frobenius theorem for
positive maps \cite{EvansHoeghKrohn1978}, the operator Jensen inequality
\cite{HansenPedersen1982}, and the symmetric doubly-stochastic scalability
of strictly positive maps \cite{KarginNonSingular2026}.  

Finally, the restriction to self-adjoint elements and real spectral
parameters can be removed by hermitization.

\begin{corollary}[Brown measure of matrix-valued semicircular elements]
\label{cor:intro-brown}
Let $a_0,a_1,\dots,a_r\in M_m(\mathbb C)$ be arbitrary and
$S=a_0\otimes1+\sum_ia_i\otimes s_i$.  For $z\in\mathbb C$ let
$\mathcal F^{\mathrm h}\subset M_{2m}(\mathbb C)$ and
$F^{\mathrm h}_z$ be the feasible set and the functional
\eqref{eq:intro-F} of the self-adjoint matrix semicircle
\[
   H_z
   =
   \begin{pmatrix}0&S-z\\ S^*-\bar z&0\end{pmatrix}
   \in M_{2m}(\mathbb C)\otimes\mathcal A,
\]
at the spectral parameter $0$.  Then, for every $z\in\mathbb C$,
\[
   \log\Delta(S-z)
   =
   -\frac12+\inf_{b\in\mathcal F^{\mathrm h}}F^{\mathrm h}_z(b)
   \qquad\text{in }[-\infty,\infty),
\]
with both sides equal to $-\infty$ if and only if $S-z$ has a nontrivial
kernel, and with the infimum attained whenever $S-z$ is invertible.  In
particular, the Brown measure of $S$ is $\frac1{2\pi}\nabla_z^2$ of the
right-hand side.
\end{corollary}

Here $\mathcal F^{\mathrm h}$ does not depend on $z$, and on its interior
$F^{\mathrm h}_z$ is a real quadratic polynomial in
$(\operatorname{Re}z,\operatorname{Im}z)$; the logarithmic potential of
the Brown measure is thus an infimum of quadratic polynomials indexed by a
fixed geodesically convex set.

\textbf{Mechanism of the proof.}
The
logarithmic determinant has an exact representation
\[
     \log\Delta(S_a-w) = \operatorname{Re}H(w),
\]
where $H$ is the Dyson action evaluated at the Dyson solution:
 \[
   H(w)
   =
   -\frac1m\log\det G(w)
   +
   \frac12\tr_m\bigl(G(w)\eta(G(w))\bigr).
\]
 A three-variable functional encodes the real and imaginary parts
of the Dyson equation; maximizing over its auxiliary variables produces the
constrained objective. To prove global minimality, we first work in the upper
half-plane with a regularized geodesically convex objective. The real-axis
formula follows by a limit of optimal values, without assuming convergence of
the minimizing scales.

\textbf{Organization of the paper.}
Section~\ref{sec:scalar-model} works out the scalar semicircle as a model
example. Section~\ref{sec:closed-formula} establishes the closed determinant
identity giving a formula for the FK determinant of the shifted matrix semicircular in terms of the Dyson action.  Section~\ref{sec:dyson-action}  relates the Dyson action to the variational problem. Section~\ref{sec:geometry} develops the geometry of the feasible
set and proves geodesic convexity of both $\mathcal F$ and $F_z$.
 Section~\ref{sec:normalizers} develops the $pUp$ factorization
\eqref{eq:intro-pup} and the balanced Dyson equations, and
Section~\ref{sec:regularized} sets and solves the regularized variational problem.
Section~\ref{sec:real-axis} proves the variational formula for all real points
(Theorem~\ref{thm:real-axis}), and
Section~\ref{sec:attainment} proves the attainment at non-degenerate regular points
 (Corollary~\ref{cor:regularized-value-limit}) and locates the
    minimizer relative to $\partial\mathcal F$
    (Proposition~\ref{prop:active-set}). Section \ref{sec:brown} treats Brown measures (Corollary~\ref{cor:brown}).  Section~\ref{sec:capacity} treats the
centered case (Theorem~\ref{thm:intro-capacity}) and Section~\ref{sec:further-questions} lists
open questions.  Appendix~\ref{app:normalizer-limits} proves the
convergence of the regularized scales at nondegenerate regular points, and
Appendix~\ref{app:local-structure} proves the local regularity of the
spectral measure.

\textbf{On the use of language models.}
This work was carried out in sustained interaction with large language
models (Claude and ChatGPT). Several of the ideas that shaped the
approach --- including the choice of balanced coordinates and the use of
geodesic convexity --- emerged in that interaction rather than from the
author alone. All statements and proofs in the paper were subsequently
worked out, checked, and rewritten by the author, who takes full
responsibility for their correctness. We record this because the paper
answers a question that was itself posed as an experiment in how such
models can contribute to mathematics.

\section{The scalar model}
\label{sec:scalar-model}

We begin by testing the variational formula in the scalar case.  Besides
checking the normalization, this example already displays the three features
that will reappear in the matrix problem: inside the spectral support the
optimizer is pinned to the boundary of $\Fset$, outside the support it moves
into the interior and satisfies the Dyson equation, and the transition occurs
exactly at the spectral edge.

Take $m=1$ and
$S_a=a+\sigma s,$
   with $\sigma>0,$
where $s$ is a standard semicircular variable.  Then
\[
   \eta(x)=\sigma^2x,
   \qquad
   \mathcal C_b(x)=b^2x.
\]
Hence the feasibility condition
$
   \mathcal C_b-\eta\succeq0
$
reduces to
$
   b^2\geq\sigma^2,
$
and therefore
\begin{equation}\label{eq:scalar-feasible}
   \Fset=[\sigma,\infty).
\end{equation}

For $b\geq\sigma$,
$
   L_b^+=b^2+\sigma^2,
$
so the variational functional becomes
\begin{equation}\label{eq:scalar-F}
   F_z(b)
   =
   \log b
   +
   \frac{(a-z)^2}{2(b^2+\sigma^2)}.
\end{equation}
Put
$
   d:=|a-z|.
$
Differentiating gives
\begin{align}
   F_z'(b)
   &=
   \frac1b
   -
   \frac{d^2b}{(b^2+\sigma^2)^2}
   \notag\\
   &=
   \frac{b}{(b^2+\sigma^2)^2}
   \left[
      \left(
         b+\frac{\sigma^2}{b}
      \right)^2-d^2
   \right].
   \label{eq:scalar-derivative}
\end{align}
The function
\[
   h(b):=b+\frac{\sigma^2}{b}
\]
is increasing on $[\sigma,\infty)$ and satisfies
$
   h(\sigma)=2\sigma.
$

\subsection*{Inside the support}

Suppose
$
   d\leq2\sigma.
$
Then, for every $b\geq\sigma$,
\[
   b+\frac{\sigma^2}{b}
   \geq
   2\sigma
   \geq
   d,
\]
and therefore \eqref{eq:scalar-derivative} gives
$
   F_z'(b)\geq0.
$
Thus the minimum is attained at the boundary point
\begin{equation}\label{eq:scalar-min-inside}
   b_*=\sigma.
\end{equation}
Consequently,
\[
   -\frac12+F_z(b_*)
   =
   \log\sigma-\frac12+\frac{(a-z)^2}{4\sigma^2}.
\]
This is precisely the logarithmic potential of the semicircle law on
$
   [a-2\sigma,a+2\sigma]
$
inside its support:
\begin{equation}\label{eq:scalar-potential-inside}
   \log\Delta(S_a-z)
   =
   \log\sigma-\frac12+\frac{(a-z)^2}{4\sigma^2},
   \qquad
   |a-z|\leq2\sigma.
\end{equation}

Thus, inside the support, the optimizer is constrained:
$
   b_*\in\partial\Fset.
$

\subsection*{Outside the support}

Now suppose
$
   d>2\sigma.
$
Then $F_z'(\sigma)<0$, while $F_z'(b)>0$ for sufficiently large $b$.
Since $h$ is strictly increasing on $(\sigma,\infty)$, there is a unique
interior critical point, determined by
\begin{equation}\label{eq:scalar-EL}
   b_*+\frac{\sigma^2}{b_*}=d.
\end{equation}
Explicitly,
\begin{equation}\label{eq:scalar-min-outside}
   b_*
   =
   \frac{
      d+\sqrt{d^2-4\sigma^2}
   }{2}
   >
   \sigma.
\end{equation}
This is the unique global minimizer of $F_z$.

Equation \eqref{eq:scalar-EL} is exactly the scalar Dyson equation in
balanced form.  Indeed, outside the support the scalar Cauchy transform is
real and satisfies
\[
   G(z)^{-1}+\sigma^2G(z)=z-a.
\]
Since
\[
   b_*=\frac1{|G(z)|},
\]
the absolute-value version of this equation is
$
   b_*+\frac{\sigma^2}{b_*}=|a-z|.
$

Using \eqref{eq:scalar-EL},
\[
   b_*^2+\sigma^2=b_*d,
\]
and therefore
\[
   -\frac12+F_z(b_*)
   =
   \log b_*
   +
   \frac{\sigma^2}{2b_*^2}.
\]
Equivalently,
\begin{equation}\label{eq:scalar-potential-outside}
\begin{aligned}
   \log\Delta(S_a-z)
   &=
   \log\left(
      \frac{
         d+\sqrt{d^2-4\sigma^2}
      }{2}
   \right)\\
   &\quad
   +
   \frac{
      d-\sqrt{d^2-4\sigma^2}
   }{
      2\bigl(d+\sqrt{d^2-4\sigma^2}\bigr)
   },
   \qquad
   d>2\sigma.
\end{aligned}
\end{equation}

Thus, outside the support, the constraint is inactive:
$
   b_*\in\operatorname{int}\Fset,
$
and the minimizer is characterized by the ordinary Euler--Lagrange
equation.

\subsection*{The spectral edge}

At
$
   d=2\sigma
$
the interior critical point collides with the boundary:
$
   b_*=\sigma.
$
Hence the spectral edge is exactly the transition between the two
variational regimes:
\[
   \begin{array}{ccl}
   |a-z| \leq 2\sigma
   &\Longleftrightarrow&
   b_*=\sigma\in\partial\Fset,\\[1mm]
   |a-z|>2\sigma
   &\Longleftrightarrow&
   b_*>\sigma,\quad b_*\in\operatorname{int}\Fset.
   \end{array}
\]

The same picture can be read directly from the Cauchy transform.  Inside
the support,
$
   |G(z+i0)|=\frac1\sigma,
$
so
\[
   b_*=\frac1{|G(z+i0)|}=\sigma.
\]
Outside the support,
$
   |G(z)|<\frac1\sigma,
$
and hence
\[
   b_*=\frac1{|G(z)|}>\sigma.
\]

This scalar example anticipates the general matrix theory.  The role of
$1/|G|$ is played by the positive factor in the $pUp$ factorization of the
matrix Dyson solution; nonzero imaginary part corresponds to an active
constraint, while in the resolvent region the optimizer moves into the
interior of the feasible set.

\section{The closed determinant identity}
\label{sec:closed-formula}

We now derive an exact expression for the logarithmic determinant directly
from the Dyson solution.  Besides being useful in the proof of the
variational formula, this identity explains the origin of the functional
$F_z$; see Section~\ref{sec:dyson-action}.

Let
\[
   G(w)
   =
   (\id\otimes\tau)
   \bigl[
      (wI-S_a)^{-1}
   \bigr],
   \qquad
   w\in\mathbb C^+.
\]
Our sign convention for the Dyson equation is
\begin{equation}\label{eq:closed-dyson}
   G(w)^{-1}+\eta(G(w))=wI-a.
\end{equation}

\begin{theorem}[Closed determinant identity]
\label{thm:closed}
For $w\in\mathbb C^+$ define
\begin{equation}\label{eq:H-def}
   H(w)
   :=
   -\frac1m\log\det G(w)
   +
   \frac12
   \tr_m\bigl(
      G(w)\eta(G(w))
   \bigr),
\end{equation}
where any holomorphic branch of $\log\det G$ on $\mathbb C^+$ may be
chosen.  Then
\begin{equation}\label{eq:H-logdet}
   \boxed{
   \operatorname{Re}H(w)
   =
   \log\Delta(S_a-w)
   }
   \qquad
   (w\in\mathbb C^+).
\end{equation}

Consequently, if $x\in\mathbb R$ is such that the boundary value
$
   G:=G(x+i0)
$
exists and is invertible, then
\begin{equation}\label{eq:closed-boundary}
   \boxed{
   \log\Delta(S_a-x)
   =
   -\frac1m\log|\det G|
   +
   \frac12
   \operatorname{Re}
   \tr_m\bigl(G\eta(G)\bigr).
   }
\end{equation}
\end{theorem}

\begin{proof}
The key observation is that $H$ is a holomorphic primitive of the scalar
Cauchy transform.

Differentiate \eqref{eq:H-def}.  Since
\[
   \frac{d}{dw}\log\det G(w)
   =
   \operatorname{Tr}\bigl(G(w)^{-1}G'(w)\bigr),
\]
we obtain
\[
\begin{aligned}
   H'(w)
   &=
   -\tr_m\bigl(G^{-1}G'\bigr)
   \\
   &\quad
   +
   \frac12
   \tr_m\bigl(
      G'\eta(G)+G\eta(G')
   \bigr).
\end{aligned}
\]
For the covariance map
\[
   \eta(X)=\sum_j A_jXA_j,
   \qquad A_j=A_j^*,
\]
cyclicity of the trace gives
\[
   \tr_m\bigl(G'\eta(G)\bigr)
   =
   \tr_m\bigl(G\eta(G')\bigr).
\]
Hence
\begin{equation}\label{eq:H-prime-1}
   H'(w)
   =
   -\tr_m(G^{-1}G')
   +
   \tr_m\bigl(G\eta(G')\bigr).
\end{equation}

Differentiating the Dyson equation \eqref{eq:closed-dyson} gives
$
   -G^{-1}G'G^{-1}
   +
   \eta(G')
   =
   I,
$
and therefore
\[
   \eta(G')
   =
   I+G^{-1}G'G^{-1}.
\]
Substituting this into \eqref{eq:H-prime-1},
\begin{align*}
   \tr_m\bigl(G\eta(G')\bigr)
   &=
   \tr_mG
   +
   \tr_m(G'G^{-1}),
\end{align*}
so the logarithmic-derivative terms cancel and
\begin{equation}\label{eq:H-prime}
   \boxed{
   H'(w)=\tr_mG(w).
   }
\end{equation}

Now define the holomorphic logarithmic potential
\[
   \Lambda(w)
   :=
   \int_{\mathbb R}
   \log(w-t)\,d\mu_{S_a}(t),
   \qquad
   w\in\mathbb C^+,
\]
using the standard branch of the logarithm on the upper half-plane.
Then
\[
   \Lambda'(w)
   =
   \int_{\mathbb R}
   \frac{1}{w-t}\,d\mu_{S_a}(t)
   =
   \tr_mG(w).
\]
Thus
$
   H-\Lambda
$
is constant on $\mathbb C^+$.

As $|w|\to\infty$ non-tangentially,
\[
   G(w)
   =
   w^{-1}I+O(w^{-2}),
\]
and therefore
\[
   H(w)
   =
   \log w+o(1)
\]
up to the purely imaginary constant coming from the chosen branch of
$\log\det G$.  Similarly,
\[
   \Lambda(w)=\log w+o(1).
\]
Hence $H-\Lambda$ is purely imaginary, and consequently
\[
   \operatorname{Re}H(w)
   =
   \operatorname{Re}\Lambda(w)
   =
   \int_{\mathbb R}
   \log|w-t|\,d\mu_{S_a}(t).
\]
This proves \eqref{eq:H-logdet}.

For the boundary formula, set $w=x+i\varepsilon$.  Then
\[
   \operatorname{Re}H(x+i\varepsilon)
   =
   \int_{\mathbb R}
   \log\sqrt{(t-x)^2+\varepsilon^2}\,
   d\mu_{S_a}(t).
\]
As $\varepsilon\downarrow0$, the right-hand side decreases to
\[
   \int_{\mathbb R}\log|t-x|\,d\mu_{S_a}(t)
   =
   \log\Delta(S_a-x),
\]
possibly with value $-\infty$.

If $G(x+i0)$ exists and is invertible, then
\[
   -\frac1m\log|\det G(x+i\varepsilon)|
   +
   \frac12
   \operatorname{Re}
   \tr_m\bigl(
      G(x+i\varepsilon)\eta(G(x+i\varepsilon))
   \bigr)
\]
converges to the right-hand side of
\eqref{eq:closed-boundary}.  This proves the boundary identity.
\end{proof}

\section{The Dyson action and the variational problem}
\label{sec:dyson-action}

Theorem~\ref{thm:intro-all-points} has a useful variational interpretation.  It is
not needed for the proofs below, but it explains why the functional $F_z$
has the form introduced in Section~\ref{sec:introduction}.

For $w\in\mathbb C^+\cup\mathbb R$ and invertible $M\in M_m(\mathbb C)$,
define the \emph{Dyson action}
\begin{equation}
\label{eq:dyson-action}
   \mathcal A_w(M)
   :=
   -\frac1m\log\det M
   -\frac12\tr_m\bigl(M\eta(M)\bigr)
   +\tr_m\bigl((wI-a)M\bigr)
   -1.
\end{equation}
The logarithm is multivalued, so $\mathcal A_w$ is defined only up to an
additive constant in $\frac{2\pi i}{m}\mathbb Z$; locally one fixes a
branch, and each branch is holomorphic in $M$.  Its differential and its
real part,
\begin{equation}
\label{eq:dyson-action-real}
   \operatorname{Re}\mathcal A_w(M)
   =
   -\frac1m\log|\det M|
   -\frac12\operatorname{Re}\tr_m\bigl(M\eta(M)\bigr)
   +\operatorname{Re}\tr_m\bigl((wI-a)M\bigr)
   -1,
\end{equation}
do not depend on the branch and are defined for every invertible $M$.
Only these are used below.  The first variation is
\[
   D_M\mathcal A_w(M)[K]
   =
   \tr_m
   \left(
      \bigl[-M^{-1}-\eta(M)+wI-a\bigr]K
   \right),
\]
so the critical-point equation for $\mathcal A_w$ is precisely
\[
   M^{-1}+\eta(M)=wI-a :
\]
the Dyson equation is the Euler--Lagrange equation of $\mathcal A_w$.

If $M=G(w)$ is the Dyson solution with $w\in\mathbb C^+$, then
$(wI-a)G(w)=I+\eta(G(w))G(w)$, and taking $\tr_m$ gives
$\tr_m\bigl((wI-a)G\bigr)=1+\tr_m\bigl(G\eta(G)\bigr)$.  Substituting
into \eqref{eq:dyson-action},
\[
   \mathcal A_w(G(w))
   =
   -\frac1m\log\det G(w)+\frac12\tr_m\bigl(G(w)\eta(G(w))\bigr)
   =H(w),
\]
where $H$ is the function of Theorem~\ref{thm:closed} and the branch of
$\log\det G(w)$ is the holomorphic one fixed there.  Hence
Theorem~\ref{thm:closed} says that the logarithmic determinant is the real
part of the \emph{on-shell Dyson action}:
\[
  \boxed{
      \log\Delta(S_a-w) = \operatorname{Re}\mathcal A_w(G(w))
   }
   \qquad
   (w\in\mathbb C^+).
\]

  The function $\KK_z(b, u, \gamma)$ that we introduced above is closely related to the Dyson action.

  \begin{proposition}[Exact identity on normalized triples]
\label{prop:action-identity}
Let $b\succ0$ and $U$ be unitary, and write
\begin{equation}\label{eq:normalized}
 M=b^{-1/2}Ub^{-1/2}=-u-i\gamma,
 \qquad u=u^*,\quad\gamma=\gamma^*.
\end{equation}
For every real $z$,
\begin{equation}\label{eq:K-action}
 \Kfun_{z}(b,u,\gamma)
 =\Rea\Act_{z}(M) + \frac12.
\end{equation}
\end{proposition}

\begin{proof}
Unitarity implies
$M^*bM=MbM^*=b^{-1}$ and $|\det M|=(\det b)^{-1}$. Adding the two
matrix identities gives
\begin{equation}\label{eq:normalization}
 b^{-1}=ubu+\gamma b\gamma,
 \qquad
 \ip{u}{\Cop{b}u}+\ip{\gamma}{\Cop{b}\gamma}=1.
\end{equation}
Direct expansion of \eqref{eq:dyson-action} gives
\begin{equation*}
 \Rea\Act_{z}(M)
 =\trm\log b+\ip{a_z}{u}
     -\frac12\ip{u}{\eta(u)}
     +\frac12\ip{\gamma}{\eta(\gamma)}-1.
\end{equation*}
In \eqref{eq:intro-KK}, the two $\Cop{b}$ terms sum to $-\frac12$ by
\eqref{eq:normalization}. This gives the same expression up to the additive constant $\frac12$ in the statement.
\end{proof}

\begin{remark}[$\KK_z$ as an extended action]
\label{rem:extended-action}

 We record three observations.

\smallskip
\noindent\emph{No minimax.}

Taking $b=tI$ with $t\downarrow0$
shows that $\inf_{b\succ0}\KK_z(b,u,\gamma)=-\infty$ for every
$(u,\gamma)$.  Hence
\[
   \sup_{\substack{u=u^*\\ \gamma\succeq0}}\ \inf_{b\succ0}
   \KK_z(b,u,\gamma)
   =-\infty ,
\]
and the order of the extrema in \eqref{eq:intro-F0} and
\eqref{eq:intro-main-formula} cannot be interchanged.

\smallskip
\noindent\emph{Common critical point and critical value.}
Consider the regularized functional
$\KK_{z,\eps}=\KK_z+\eps\trm\gamma$ of Section~\ref{sec:regularized}.
The computation in the proof of
Proposition~\ref{prop:action-identity} gives
$\KK_{z,\eps}(b,u,\gamma)=\Rea\Act_{z+i\eps}(M)+\frac12$ on normalized
triples.  The stationarity conditions of $\KK_{z,\eps}$ in $u$, $\gamma$
and $b$ are
\[
   L_b^+u=a_z,
   \qquad
   L_b^-\gamma=\eps I,
   \qquad
   b^{-1}=ubu+\gamma b\gamma ,
\]
which are the balanced Dyson equations \eqref{eq:balanced}.  

Eliminating a group of variables through its own stationarity condition,
when that condition determines them uniquely and smoothly, changes neither
the critical points nor the critical values of a function.  For
$b\in\Fint$, eliminating $(u,\gamma)$ produces the reduced objective
\[
   F_{z,\eps}(b)=\max_{u,\gamma}\KK_{z,\eps}(b,u,\gamma).
\]
For $\gamma\succ0$, the function $b\mapsto\KK_{z,\eps}(b,u,\gamma)$ is
strictly geodesically concave: $\tr_m\log b$ is affine along geodesics,
and by Lemma~\ref{lem:C-geodesic} the second derivative of
$\langle\gamma,\mathcal C_{b_t}\gamma\rangle$ at $t=0$ is
$\|Xw+wX\|_2^2$ with $w=b^{1/2}\gamma b^{1/2}\succ0$, which is positive
for $X\neq0$.  Its unique critical point is the unitary normalizer of
$-u-i\gamma$, so by Proposition~\ref{prop:action-identity}
\[
   \max_{b\succ0}\KK_{z,\eps}(b,u,\gamma)
   =\Rea\Act_{z+i\eps}(-u-i\gamma)+\frac12 .
\]
Thus $F_{z,\eps}$ and $\Rea\Act_{z+i\eps}+\frac12$ are two reductions of
the same function, and the equality of their critical values is automatic.

\smallskip
\noindent\emph{Type of the critical point.}
What is not automatic is the nature of the critical point.  Since
$\Act_w$ is holomorphic in $M$, a nondegenerate critical point of
$\Rea\Act_w$ has as many negative as positive directions, so the Dyson
action itself admits no extremal characterization of the Dyson solution.
In the extended functional the negative directions are carried by the
auxiliary variables $(u,\gamma)$ and are removed by the inner
maximization.  The remaining problem in $b$ is geodesically convex
(Corollary~\ref{cor:Freg-convex}), and its critical point is a global
minimizer.  This is the reason for introducing the scale $b$ as an
independent variable.
\end{remark}

\section{Geometry of the variational problem}
\label{sec:geometry}

In this section we study the finite-dimensional optimization problem
introduced in Section~\ref{sec:introduction}.  No properties of the Dyson
solution are used here.  The two conclusions needed later are that the
feasible set $\Fset$ is geodesically convex in the positive cone and that
$F_z$ is geodesically convex on $\operatorname{int}\Fset$.

Throughout, $\|\cdot\|$ denotes the operator norm on Hilbert--Schmidt
space.

\subsection{The feasible set and the normalized covariance map}
\label{subsec:feasible-set}

For $b\succ0$ define the normalized covariance map
\begin{equation}\label{eq:eta-hat}
   \widehat\eta_b
   :=
   \mathcal C_{b^{-1/2}}
   \eta
   \mathcal C_{b^{-1/2}}.
\end{equation}
Equivalently,
\[
   \widehat\eta_b(X)
   =
   \sum_{j=1}^r
   \widehat A_jX\widehat A_j,
   \qquad
   \widehat A_j
   :=
   b^{-1/2}A_jb^{-1/2}.
\]
Thus $\widehat\eta_b$ is completely positive and Hilbert--Schmidt
self-adjoint.

The apparently one-sided constraint defining $\Fset$ is in fact
two-sided.

\begin{lemma}[Perron--Frobenius description of the feasible set]
\label{lem:feasible-pf}
For every $b\succ0$, the normalized covariance map
$\widehat\eta_b$ has a nonzero positive semidefinite eigenmatrix
$P_b$ at its spectral radius:
\begin{equation}\label{eq:positive-Perron-eigenmatrix}
   \widehat\eta_b(P_b)=r_bP_b,
   \qquad
   P_b\succeq0,\quad P_b\ne0,
   \qquad
   r_b
   =
   \|\widehat\eta_b\|_{2\to2}
   =
   \lambda_{\max}(\widehat\eta_b).
\end{equation}
For $b\succ0$, the following are equivalent:
\begin{enumerate}
\item[(i)] $b\in\Fset$;
\item[(ii)] $\lambda_{\max}(\widehat\eta_b)\leq1$;
\item[(iii)] $\|\widehat\eta_b\|\leq1$;
\item[(iv)]
$L_b^- \succeq0$ 
   and
 $  L_b^+ \succeq0$ as Hilbert-Shmidt operators. 
\end{enumerate}
Moreover,
\begin{equation}\label{eq:interior-feasible}
   b\in\operatorname{int}\Fset
   \quad\Longleftrightarrow\quad
   \|\widehat\eta_b\|<1.
\end{equation}
In particular,
$L_b^- \succ 0$ and $L_b^+\succ0$ for $b\in\FF^\circ$.
\end{lemma}

\begin{proof}
Let
$
   T_b:=\mathcal C_{b^{-1/2}}.
$
As an operator on Hilbert--Schmidt space, $T_b$ is self-adjoint, positive
definite, and invertible.  A direct calculation gives
$
   T_b\mathcal C_bT_b=I,
$
and therefore
\begin{equation}\label{eq:normalized-minus}
   T_b(\mathcal C_b-\eta)T_b
   =
   I-\widehat\eta_b.
\end{equation}
Since a congruence $A\mapsto T_bAT_b$ by an invertible positive operator
preserves positivity of self-adjoint operators,
\[
   \mathcal C_b-\eta\succeq0
   \quad\Longleftrightarrow\quad
   I-\widehat\eta_b\succeq0.
\]
Because $\widehat\eta_b$ is Hilbert--Schmidt self-adjoint, the latter is
equivalent to
$
   \lambda_{\max}(\widehat\eta_b)\leq1.
$

We next compare the largest eigenvalue with the operator norm.
The map $\widehat\eta_b$ is positive as a map on matrices, so the
Perron--Frobenius theorem for positive maps \cite{EvansHoeghKrohn1978} implies that its spectral
radius
$
   r(\widehat\eta_b)
$
is an eigenvalue with a positive semidefinite eigenvector.  Since
$\widehat\eta_b$ is Hilbert--Schmidt self-adjoint, its spectrum is real
and
\[
   \|\widehat\eta_b\|
   =
   r(\widehat\eta_b)
   =
   \lambda_{\max}(\widehat\eta_b).
\]
This proves the equivalence of (ii) and (iii).

For a self-adjoint operator, $\|\widehat\eta_b\|\leq1$ is equivalent to
$
   -I\preceq\widehat\eta_b\preceq I,
$
or
$
   I\pm\widehat\eta_b\succeq0.
$
Applying the inverse congruence to
\[
   T_b(\mathcal C_b\pm\eta)T_b
   =
   I\pm\widehat\eta_b
\]
gives (iv).

It remains to identify the interior.  If
$\|\widehat\eta_b\|<1$, continuity in $b$ shows that the same strict
inequality holds in a neighborhood of $b$, so $b\in\operatorname{int}\Fset$.
Conversely, if $\|\widehat\eta_b\|=1$, then for $0<\lambda<1$,
$
   \widehat\eta_{\lambda b}
   =
   \lambda^{-2}\widehat\eta_b,
$
and hence
\[
   \|\widehat\eta_{\lambda b}\|
   =
   \lambda^{-2}>1.
\]
Thus points outside $\Fset$ occur arbitrarily close to $b$, so $b$ is a
boundary point.  This proves \eqref{eq:interior-feasible}.
\end{proof}

Maximization over $\gamma$ confines $b$ to the feasible set $\FF$. 

\begin{proposition}[Eliminating the multiplier enforces feasibility]
\label{prop:indicator}
For every $b\succ0$,
\begin{equation}\label{eq:indicator}
 \sup_{\gamma\succeq0}
 \left\{-\frac12\ip{\gamma}{L_b^-\gamma}\right\}
 =\begin{cases}
   0,&b\in\Fset,\\
   +\infty,&b\notin\Fset.
  \end{cases}
\end{equation}
Consequently,
\begin{equation}\label{eq:three-variable-real}
 \inf_{b\succ0}\sup_{\substack{u=u^*\\\gamma\succeq0}}
    \Kfun_z(b,u,\gamma)
 =\inf_{b\in\Fset}F_z(b).
\end{equation}
\end{proposition}

\begin{proof}
If $b\in\Fset$, then $L_b^-=\mathcal C_b-\eta\succeq0$ on the whole
Hilbert--Schmidt space, so every quantity in the supremum is
nonpositive, and $\gamma=0$ gives the value $0$.

Let $b\notin\Fset$.  By Lemma~\ref{lem:feasible-pf},
$\|\widehat\eta_b\|>1$.  The map $\widehat\eta_b$ is positive, so by the
Perron--Frobenius theorem for positive maps \cite{EvansHoeghKrohn1978}
its spectral radius $r$ is an eigenvalue with a positive semidefinite
eigenvector: there is $Q\succeq0$, $Q\neq0$, with
$\widehat\eta_b(Q)=rQ$.  Since $\widehat\eta_b$ is Hilbert--Schmidt
self-adjoint, $r=\|\widehat\eta_b\|>1$.

Put $\gamma_0:=b^{-1/2}Qb^{-1/2}\succeq0$.  Then
$\mathcal C_b(\gamma_0)=b^{1/2}Qb^{1/2}$, and by the definition
\eqref{eq:eta-hat} of $\widehat\eta_b$,
\[
 \ip{\gamma_0}{\mathcal C_b\gamma_0}=\ip{Q}{Q},
 \qquad
 \ip{\gamma_0}{\eta(\gamma_0)}
 =\ip{Q}{\widehat\eta_b(Q)}
 =r\ip{Q}{Q}.
\]
Hence $\ip{\gamma_0}{L_b^-\gamma_0}=(1-r)\ip{Q}{Q}<0$.  For $t>0$ we
have $t\gamma_0\succeq0$ and
\[
 -\frac12\ip{t\gamma_0}{L_b^-(t\gamma_0)}
 =\frac{t^2}{2}(r-1)\ip{Q}{Q}\longrightarrow+\infty
 \qquad(t\to\infty).
\]
This proves \eqref{eq:indicator}.

For fixed $b\succ0$, $\Kfun_z(b,u,\gamma)$ is the sum of $\trm\log b$,
a function of $u$ alone, and a function of $\gamma$ alone.  Therefore
\[
 \sup_{\substack{u=u^*\\\gamma\succeq0}}\Kfun_z(b,u,\gamma)
 =\trm\log b
  +\sup_{u=u^*}\Bigl\{\ip{a_z}{u}-\frac12\ip{u}{L_b^+u}\Bigr\}
  +\sup_{\gamma\succeq0}\Bigl\{-\frac12\ip{\gamma}{L_b^-\gamma}\Bigr\}.
\]
The supremum over $u$ is at least $0$ (take $u=0$), so the right-hand
side is well defined in $(-\infty,+\infty]$.  By \eqref{eq:indicator}
and \eqref{eq:intro-F}, it equals $F_z(b)$ for $b\in\Fset$ and $+\infty$
for $b\notin\Fset$.  Taking the infimum over $b\succ0$ gives
\eqref{eq:three-variable-real}; the order of the extrema is unchanged.
\end{proof}

\subsection{The boundary values}
\label{subsec:legendre}

The following elementary approximation will be useful whenever a minimizer
lies on the boundary of $\Fset$.

\begin{lemma}[Radial approximation of boundary values]
\label{lem:radial-approximation}
Let $b\in\Fset$.  Then
$
   \lambda b\in\FF^\circ$ for every $\lambda>1$,
and
\begin{equation}\label{eq:radial-limit}
   F_z(\lambda b)
   \longrightarrow
   F_z(b)
   \qquad
   (\lambda\downarrow1)
\end{equation}
in the extended-real sense.
\end{lemma}

\begin{proof}
Since
$
   \mathcal C_{\lambda b}
   =
   \lambda^2\mathcal C_b,
$
we have
\[
   \mathcal C_{\lambda b}-\eta
   =
   (\lambda^2-1)\mathcal C_b
   +
   (\mathcal C_b-\eta)
   \succ0.
\]
Thus $\lambda b\in\operatorname{int}\Fset$.

Write
$
   F_z(b)=\trm\log b+Q_z(b),
$
where
\[
   Q_z(b)
   :=
   \sup_{u=u^*}
   \left\{
      \ip{a_z}{u}-\frac12\ip{u}{L_b^+ u}
   \right\}
   \in[0,+\infty].
\]
  Since
\[
   L_{\lambda b}^+
   =
   L_b^+ +(\lambda^2-1)\mathcal C_b,
\]
we have
\begin{align*}
   Q_z(\lambda b)
   =
   \sup_{u=u^*}
   \biggl\{
      \ip{a_z}{u}
      -
      \frac12\ip{u}{L_b^+ u}
      -
      \frac{\lambda^2-1}{2}
      \ip{u}{\mathcal C_bu}
   \biggr\}.
\end{align*}

For each fixed $u$, the expression in braces increases to
$\ip{a_z}{u}-\frac12\ip{u}{L_b^+u}$ as $\lambda\downarrow1$, because
$\ip{u}{\mathcal C_bu}\ge0$.  Hence $Q_z(\lambda b)$ is nondecreasing as
$\lambda\downarrow1$, its limit is the supremum over $\lambda>1$, and the
two suprema may be interchanged:
\[
   \lim_{\lambda\downarrow1}Q_z(\lambda b)
   =\sup_{\lambda>1}\,\sup_{u=u^*}\{\cdots\}
   =\sup_{u=u^*}\,\sup_{\lambda>1}\{\cdots\}
   =Q_z(b)
\]
in the extended-real sense, including the case $Q_z(b)=+\infty$.
Finally,
\[
   \trm\log(\lambda b)
   =
   \log\lambda+\trm\log b
   \longrightarrow
   \trm\log b,
\]
which proves \eqref{eq:radial-limit}.
\end{proof}

\subsection{Geodesic convexity of the feasible set}
\label{subsec:geodesic-feasible}

We use the standard affine-invariant geometry on the positive cone.  The
geodesic joining $b_0$ to $b_1$ is
\[
   b_t
   :=
   b_0\#_tb_1
   =
   b_0^{1/2}
   \bigl(
      b_0^{-1/2}b_1b_0^{-1/2}
   \bigr)^t
   b_0^{1/2},
   \qquad
   0\leq t\leq1.
\]
Writing
$
   b_1=b_0^{1/2}e^Xb_0^{1/2},$ with $X=X^*$,
this becomes
$
   b_t=b_0^{1/2}e^{tX}b_0^{1/2}.
$


We observe that the congruence representation
\[
   b\longmapsto\mathcal C_b,
   \qquad
   \mathcal C_b(Y):=bYb,
\]
is compatible with the affine-invariant geodesics of the positive cone.

Indeed, let
$
   B=M_m(\mathbb C)
$
with its Hilbert--Schmidt inner product.  Denote by
$\operatorname{PD}(B)$ the cone of positive definite self-adjoint
linear operators on the Hilbert space $B$.
For $b\succ0$, recall that
$
   \mathcal C_b:B\to B,$
   $\mathcal C_b(Y)=bYb,$
belongs to $\operatorname{PD}(B)$.

\begin{lemma}[$\mathcal C$ maps geodesics to geodesics]
\label{lem:C-geodesic}
Let
$
   b_t=b^{1/2}e^{tX}b^{1/2},$
   with
   $X=X^*,$
and define the self-adjoint operator on Hilbert--Schmidt space $B$:
\[
   D_X(Y):=XY+YX.
\]
Then
\begin{equation}\label{eq:C-geodesic}
   \mathcal C_{b_t}
   =
   \mathcal C_{b^{1/2}}\,
   e^{tD_X}\,
   \mathcal C_{b^{1/2}}.
\end{equation}
Hence $t\mapsto\mathcal C_{b_t}$ is a geodesic in the cone of positive
definite operators on Hilbert--Schmidt space $B$.  In particular,
\begin{equation}\label{eq:C-geodesic-equation}
   \mathcal C_{b_t}''
   =
   \mathcal C_{b_t}'
   \mathcal C_{b_t}^{-1}
   \mathcal C_{b_t}'.
\end{equation}
\end{lemma}

\begin{proof}
Left and right multiplication by $X$ commute, and therefore
\[
   e^{tD_X}(Y)
   =
   e^{tX}Ye^{tX}
   =
   \mathcal C_{e^{tX}}(Y).
\]
Since
$
   \mathcal C_{b_t}
   =
   \mathcal C_{b^{1/2}}
   \mathcal C_{e^{tX}}
   \mathcal C_{b^{1/2}},
$
this proves \eqref{eq:C-geodesic}.

Write
$
   P:=\mathcal C_{b^{1/2}},$ and 
   $
   D:=D_X.
$
Then
$
   \mathcal C_{b_t}=Pe^{tD}P,
$
so
\[
   \mathcal C_{b_t}'
   =
   PDe^{tD}P,
   \qquad
   \mathcal C_{b_t}''
   =
   PD^2e^{tD}P.
\]
Since
$
   \mathcal C_{b_t}^{-1}
   =
   P^{-1}e^{-tD}P^{-1},
$
we obtain
\begin{align*}
   \mathcal C_{b_t}'
   \mathcal C_{b_t}^{-1}
   \mathcal C_{b_t}'
   &=
   PDe^{tD}P
   P^{-1}e^{-tD}P^{-1}
   PDe^{tD}P  \\
   &=
   PD^2e^{tD}P
   =
   \mathcal C_{b_t}''.
\end{align*}
\end{proof}

\begin{lemma}[Geodesic log-convexity of the normalized covariance]
\label{lem:geodesic-feasible}
Let $b_t=b_0\#_tb_1$.  Then
\begin{equation}\label{eq:eta-geodesic-logconvex}
   \|\widehat\eta_{b_t}\|
   \leq
   \|\widehat\eta_{b_0}\|^{1-t}
   \|\widehat\eta_{b_1}\|^t.
\end{equation}
Consequently, $\Fset$ is geodesically convex.  If
$b_0\in\Fset$ and $b_1\in\operatorname{int}\Fset$, then
\[
   b_t\in\operatorname{int}\Fset
   \qquad
   \text{for every }t\in(0,1].
\]
\end{lemma}

\begin{proof}
Let
\[
   W:=e^{D_X}\succ0,
   \qquad
   \Phi:=\widehat\eta_{b_0}.
\]
By \eqref{eq:C-geodesic},
$
   \mathcal C_{b_t}
   =
   \mathcal C_{b_0^{1/2}}
   W^t
   \mathcal C_{b_0^{1/2}}.
$

For $b\succ0$, introduce the weighted norm
\[
   \|Y\|_{\mathcal C_b}^2
   :=
   \ip{Y}{\mathcal C_bY}.
\]
A change of variables in the Hilbert--Schmidt operator norm gives
\begin{equation}\label{eq:eta-quotient}
   \|\widehat\eta_b\|
   =
   \sup_{Y,Z\neq0}
   \frac{
      |\ip{Y}{\eta(Z)}|
   }{
      \|Y\|_{\mathcal C_b}
      \|Z\|_{\mathcal C_b}
   }.
\end{equation}
Applying this to $b_t$ and changing variables by
$\mathcal C_{b_0^{-1/2}}$ gives
\[
   \|\widehat\eta_{b_t}\|
   =
   \bigl\|
      W^{-t/2}\Phi W^{-t/2}
   \bigr\|.
\]

Now consider
\[
   f(\zeta)
   :=
   W^{-\zeta/2}\Phi W^{-\zeta/2},
   \qquad
   0\leq\operatorname{Re}\zeta\leq1.
\]
Since $W\succ0$, the complex powers are well defined and $f$ is analytic
in the strip.  On the line $\operatorname{Re}\zeta=0$,
$W^{-is/2}$ is unitary, so
\[
   \|f(is)\|=\|\Phi\|=\|\widehat\eta_{b_0}\|.
\]
Similarly,
\[
   \|f(1+is)\|
   =
   \|W^{-1/2}\Phi W^{-1/2}\|
   =
   \|\widehat\eta_{b_1}\|.
\]

Fix unit vectors $u, v$ in Hilbert--Schmidt space and apply
Hadamard's three-lines theorem to the scalar analytic function
\[
   \zeta\longmapsto
   \ip{u}{f(\zeta)v}.
\]
Taking the supremum over the two unit vectors gives
\eqref{eq:eta-geodesic-logconvex}.

If $b_0,b_1\in\Fset$, then both endpoint norms are at most $1$, hence
$
   \|\widehat\eta_{b_t}\|\leq1,
$
and $b_t\in\Fset$ by Lemma~\ref{lem:feasible-pf}.  If in addition
$b_1\in\operatorname{int}\Fset$, then
$\|\widehat\eta_{b_1}\|<1$, and for every $t>0$,
$
   \|\widehat\eta_{b_t}\|
   \leq
   \|\widehat\eta_{b_1}\|^t
   <1.
$
Thus $b_t\in\operatorname{int}\Fset$.
\end{proof}

\subsection{Geodesic convexity of the objective}
\label{subsec:geodesic-F}

In addition to the convexity of $F_z$, we record the corresponding
operator convexity along geodesics for both inverse terms
$(\mathcal C_b+\eta)^{-1}$ and $(\mathcal C_b-\eta)^{-1}$.
The second will be used later. Throughout this subsection, write
\begin{equation*}
   L_b^+:=\mathcal C_b+\eta=L_b,
   \qquad
   L_b^-:=\mathcal C_b-\eta.
\end{equation*}
All inequalities between operators are understood in the
Hilbert--Schmidt quadratic-form order.

\begin{lemma}[Inverse of a shifted operator geodesic]
\label{lem:resolvent-geodesic}
Let $t\mapsto\mathcal C_t$ be a geodesic in $\operatorname{PD}(B)$,
and let $\Theta=\Theta^*$ be a fixed operator on $B$. Suppose
$L_t:=\mathcal C_t+\Theta\succ0$ throughout the interval under
consideration. Define
\begin{equation*}
   \widehat\Theta_t
   :=\mathcal C_t^{-1/2}\Theta\mathcal C_t^{-1/2},
   \qquad
   K_t:=\mathcal C_t^{-1/2}\mathcal C_t'L_t^{-1}.
\end{equation*}
Then
\begin{equation}\label{eq:resolvent-hessian}
   \frac{d^2}{dt^2}L_t^{-1}=K_t^*\Gamma_tK_t,
\end{equation}
where
\begin{equation}\label{eq:Cayley-factor}
   \Gamma_t
   :=(I-\widehat\Theta_t)(I+\widehat\Theta_t)^{-1}.
\end{equation}
In particular, at every $t$ for which $\Theta\preceq\mathcal C_t$,
\begin{equation*}
   \frac{d^2}{dt^2}L_t^{-1}\succeq0.
\end{equation*}
No positivity assumption on the fixed operator $\Theta$ is required.
\end{lemma}

\begin{proof}
Put $R_t:=L_t^{-1}$. Differentiating the inverse gives
\begin{equation*}
   R_t'=-R_t\mathcal C_t'R_t,
   \qquad
   R_t''=2R_t\mathcal C_t'R_t\mathcal C_t'R_t
            -R_t\mathcal C_t''R_t.
\end{equation*}
The geodesic equation
$\mathcal C_t''=\mathcal C_t'\mathcal C_t^{-1}\mathcal C_t'$
therefore yields
\begin{equation}\label{eq:resolvent-second-factor}
   R_t''
   =R_t\mathcal C_t'
      \bigl(2R_t-\mathcal C_t^{-1}\bigr)
      \mathcal C_t'R_t.
\end{equation}
Since
\begin{equation*}
   L_t=\mathcal C_t^{1/2}
           (I+\widehat\Theta_t)\mathcal C_t^{1/2},
\end{equation*}
we have
\begin{align}
   2L_t^{-1}-\mathcal C_t^{-1}
   &=\mathcal C_t^{-1/2}
       \bigl[2(I+\widehat\Theta_t)^{-1}-I\bigr]
       \mathcal C_t^{-1/2}\notag\\
   &=\mathcal C_t^{-1/2}
       (I-\widehat\Theta_t)(I+\widehat\Theta_t)^{-1}
       \mathcal C_t^{-1/2}.
   \label{eq:Cayley-factorization}
\end{align}
Substitution into \eqref{eq:resolvent-second-factor} proves
\eqref{eq:resolvent-hessian}.

The assumption $L_t\succ0$ implies $I+\widehat\Theta_t\succ0$.
If also $\Theta\preceq\mathcal C_t$, then
$I-\widehat\Theta_t\succeq0$. These two operators commute, since
both are functions of the self-adjoint operator $\widehat\Theta_t$.
Hence $\Gamma_t\succeq0$ by functional calculus, and
\eqref{eq:resolvent-hessian} gives the claimed inequality.
\end{proof}

\begin{corollary}[Geodesic convexity of both inverse terms]
\label{cor:two-inverse-geodesic}
Let $b_t$ be a geodesic contained in
$\operatorname{int}\mathcal F$. Put
\begin{equation*}
   \mathcal C_t:=\mathcal C_{b_t},
   \qquad
   R_\pm(t):=(L_{b_t}^\pm)^{-1},
   \qquad
   K_\pm(t):=\mathcal C_t^{-1/2}\mathcal C_t'R_\pm(t).
\end{equation*}
Then, with
$\widehat\eta_t:=\mathcal C_t^{-1/2}\eta\mathcal C_t^{-1/2}
=\widehat\eta_{b_t}$,
\begin{equation}\label{eq:two-inverse-hessian}
   R_\pm''(t)
   =K_\pm(t)^*
      (I\mp\widehat\eta_t)(I\pm\widehat\eta_t)^{-1}
      K_\pm(t)
   \succeq0.
\end{equation}
Thus both $b\mapsto(L_b^+)^{-1}$ and $b\mapsto(L_b^-)^{-1}$
are geodesically convex in the L\"owner order on
$\operatorname{int}\mathcal F$. In particular, for every fixed
Hermitian matrix $v$, the following functions are geodesically convex there:
\begin{equation}\label{eq:two-inverse-quadratic}
   q_v^\pm(b):=\frac12
      \left\langle v,(L_b^\pm)^{-1}v\right\rangle.
\end{equation}
\end{corollary}

\begin{proof}
By Lemma~\ref{lem:C-geodesic}, $t\mapsto\mathcal C_t$ is an
operator geodesic. By Lemma~\ref{lem:feasible-pf},
$\|\widehat\eta_t\|<1$, so $I\pm\widehat\eta_t\succ0$ and
$L_{b_t}^\pm\succ0$.

Apply Lemma~\ref{lem:resolvent-geodesic} with
$\Theta=\eta$ and $\Theta=-\eta$, respectively. This gives
\eqref{eq:two-inverse-hessian}. Both middle factors are positive
definite: for $-1<\lambda<1$, the numbers
$(1-\lambda)/(1+\lambda)$ and $(1+\lambda)/(1-\lambda)$ are positive.
Pairing \eqref{eq:two-inverse-hessian} with the fixed vector $v$
proves the assertion about $q_v^\pm$.

The use of $\Theta=-\eta$ requires only self-adjointness of the
shift in the abstract lemma; it does not assert that $-\eta$ is a
positive map.
\end{proof}

\begin{proposition}[Geodesic Hessian of $F_z$]
\label{prop:geodesic-hessian}
Let $b\in\operatorname{int}\mathcal F$, let $X=X^*$, and consider
$b_t=b^{1/2}e^{tX}b^{1/2}$ for $t$ sufficiently close to zero.
Set
\begin{equation*}
   u:=(L_b^+)^{-1}a_z,
   \qquad
   \widetilde u:=b^{1/2}ub^{1/2}.
\end{equation*}
Then
\begin{equation}\label{eq:geodesic-hessian}
   \boxed{
   \left.\frac{d^2}{dt^2}F_z(b_t)\right|_{t=0}
   =\frac12\left\langle
      X\widetilde u+\widetilde uX,\,
      (I-\widehat\eta_b)(I+\widehat\eta_b)^{-1}
      (X\widetilde u+\widetilde uX)
   \right\rangle.
   }
\end{equation}
In particular, this second derivative is nonnegative, and $F_z$
is geodesically convex on $\mathcal F$.
\end{proposition}

\begin{proof}
On $\operatorname{int}\mathcal F$,
\begin{equation*}
   F_z(b)=\operatorname{tr}_m\log b+q_{a_z}^+(b).
\end{equation*}
The quadratic term is geodesically convex by
Corollary~\ref{cor:two-inverse-geodesic}. The logarithmic term is
affine along geodesics:
\begin{equation*}
   \operatorname{tr}_m\log b_t
   =\operatorname{tr}_m\log b
      +t\,\operatorname{tr}_mX.
\end{equation*}
This proves geodesic convexity of $F_z$.

For the explicit Hessian, put
$h:=\dot b_0=b^{1/2}Xb^{1/2}$ and
$\mathcal C_t:=\mathcal C_{b_t}$. Then
\begin{equation*}
   \mathcal C_0'(Y)=hYb+bYh.
\end{equation*}
The logarithmic term has zero second derivative, so the plus case
of \eqref{eq:two-inverse-hessian} gives
\begin{equation*}
   \left.\frac{d^2}{dt^2}F_z(b_t)\right|_{t=0}
   =\frac12\left\langle
      K_+(0)a_z,\,
      (I-\widehat\eta_b)(I+\widehat\eta_b)^{-1}K_+(0)a_z
   \right\rangle.
\end{equation*}
Since $\mathcal C_b^{-1/2}=\mathcal C_{b^{-1/2}}$, we obtain
\begin{align*}
   K_+(0)a_z
   &=\mathcal C_b^{-1/2}\mathcal C_0'u\\
   &=b^{-1/2}(hub+buh)b^{-1/2}\\
   &=Xb^{1/2}ub^{1/2}+b^{1/2}ub^{1/2}X\\
   &=X\widetilde u+\widetilde uX.
\end{align*}
Substitution proves \eqref{eq:geodesic-hessian}. 

Since this second derivative is nonnegative, $F_z$ is geodesically convex on $\Fint$. For the points on the boundary, we can approximate: $(\lambda b_0) \#_t (\lambda b_1) = \lambda (b_0 \#_t b_1)$ is interior for $\lambda > 1$; let $\lambda \downarrow 1$ with Lemma \ref{lem:radial-approximation}.

\end{proof}

\section{Unitary normalizers and balanced Dyson equations}
\label{sec:normalizers}

The variational problem is formulated in terms of a positive matrix $b$.
We now explain how such a matrix arises canonically from the Dyson
solution itself.

The relevant factorization is a symmetric analogue of the polar
decomposition.

\subsection{Unitary normalizers and the $pUp$ factorization}

\begin{definition}[Unitary normalizer]
\label{def:unitary-normalizer}
Let $G\in M_m(\mathbb C)$ be invertible.  A matrix $b\succ0$ is called a
\emph{unitary normalizer} of $G$ if
\begin{equation}\label{eq:unitary-normalizer}
   U:=b^{1/2}Gb^{1/2}
\end{equation}
is unitary.
\end{definition}

The condition in Definition~\ref{def:unitary-normalizer} has several useful
equivalent forms.  Since $U^*U=UU^*=I$,
\begin{equation}\label{eq:normalizer-identities}
   G^*bG=b^{-1},
   \qquad
   GbG^*=b^{-1}.
\end{equation}
Equivalently,
\begin{equation}\label{eq:normalizer-inverse}
   G^{-1}=bG^*b.
\end{equation}

If $G=G^*$ is invertible, then
$
   b=|G|^{-1}
$
is a unitary normalizer and
\[
   b^{1/2}Gb^{1/2}
   =
   \operatorname{sign}(G),
\]
a self-adjoint unitary.  Thus the usual sign decomposition of a
self-adjoint matrix is the real special case of the construction below.

\begin{lemma}[$pUp$ factorization]
\label{lem:pup}
Let
\[
   G=X-iY,
   \qquad
   X=X^*,
   \qquad
   Y\succ0.
\]
Set
\[
   T:=Y^{-1/2}XY^{-1/2}
\]
and define $b_*\succ0$ by
\begin{equation}\label{eq:pup-b}
   b_*^{-1}
   :=
   Y^{1/2}
   \bigl(I+T^2\bigr)^{1/2}
   Y^{1/2}.
\end{equation}
Then $b_*$ is a unitary normalizer of $G$.

Equivalently, every matrix with strictly negative imaginary part admits a
factorization
\begin{equation}\label{eq:pup}
   G=pUp,
   \qquad
   p\succ0,
   \qquad
   U\in U(m),
\end{equation}
with $p=b_*^{-1/2}$. The factorization is unique: if $b$ is a unitary normalizer of $G$, then $b=b_*$.

For $m=1$,
$
   b_*=|G|^{-1}.
$
\end{lemma}

This factorization is closely related to the balanced polar
decomposition used in the analysis of the matrix Dyson equation;
see \cite{AjankiErdosKruger2019}.

\begin{proof}
Since
$
   G
   =
   Y^{1/2}(T-iI)Y^{1/2},
$
it is enough first to factor $T-iI$.

Because $T=T^*$, the matrices $T$ and $I+T^2$ commute.  Define
\[
   A:=(I+T^2)^{1/4}
\]
and
\[
   U_0
   :=
   (T-iI)(I+T^2)^{-1/2}.
\]
Then
\[
   T-iI
   =
   A\,U_0\,A.
\]
Moreover,
\begin{align*}
   U_0^*U_0
   &=
   (I+T^2)^{-1/2}
   (T+iI)(T-iI)
   (I+T^2)^{-1/2} \\
   &=
   I,
\end{align*}
so $U_0$ is unitary.

Hence
$
   G
   =
   L\,U_0\,L^*,$ where
  $ L:=Y^{1/2}A. $
Take the polar decomposition
\[
   L=pV,
   \qquad
   p:=(LL^*)^{1/2}\succ0,
   \qquad
   V\in U(m).
\]
Then
\[
   G
   =
   p\,(VU_0V^*)\,p,
\]
and $VU_0V^*$ is unitary.  Thus \eqref{eq:pup} holds.

Finally,
\[
   p^2
   =
   LL^*
   =
   Y^{1/2}
   (I+T^2)^{1/2}
   Y^{1/2}.
\]
Since $p=b_*^{-1/2}$, this is precisely
\eqref{eq:pup-b}.

For uniqueness, let $b$ be any unitary normalizer of $G$.  By
\eqref{eq:normalizer-inverse}, $G^{-1}=bG^*b=b(X+iY)b$, so
\[
   bYb=\operatorname{Im}\bigl(G^{-1}\bigr)
   :=\frac{1}{2i}\bigl(G^{-1}-(G^{-1})^*\bigr).
\]
Since $Y\succ0$, the right-hand side is positive definite, and the
equation $bYb=N$ with $Y,N\succ0$ has the unique positive definite
solution $b=Y^{-1/2}\bigl(Y^{1/2}NY^{1/2}\bigr)^{1/2}Y^{-1/2}$.
\end{proof}

\begin{remark}
The formula \eqref{eq:pup-b} is the noncommutative analogue of the scalar
identity
\[
   |x-iy|
   =
   y
   \sqrt{
      1+\left(\frac{x}{y}\right)^2
   }.
\]
Thus $b_*^{-1}$ plays the role of the modulus of $G$, but in a form adapted
to symmetric congruence rather than to the ordinary polar decomposition.
\end{remark}

\subsection{Real and imaginary parts of the Dyson equation}

We now combine the unitary-normalizer identity with the Dyson equation.
This elementary observation will be used repeatedly.

\begin{lemma}[Real and imaginary parts of the balanced Dyson equation]
\label{lem:balanced-dyson-parts}
Let
$
   w=x+i\varepsilon,$
   with
   $\varepsilon\geq0,
$
and suppose that
\[
   G:=G(w)
   =
   -u-i\gamma,
   \qquad
   u=u^*,
   \qquad
   \gamma=\gamma^*\succeq0,
\]
is invertible.  Let $b\succ0$ be a unitary normalizer of $G$.  Then
\begin{align}
   (\mathcal C_b+\eta)(u)
   &=
   a_x,
   \label{eq:dyson-real-unbalanced}\\
   (\mathcal C_b-\eta)(\gamma)
   &=
   \varepsilon I,
   \label{eq:dyson-imag-unbalanced}
\end{align}
where
$
   a_x:=a-xI.
$
\end{lemma}

\begin{proof}
Since $b$ is a unitary normalizer,
$
   G^{-1}=bG^*b.
$
Using
\[
   G=-u-i\gamma,
   \qquad
   G^*=-u+i\gamma,
\]
we obtain
\begin{equation}\label{eq:normalizer-parts}
   G^{-1}
   =
   -bub+i\,b\gamma b.
\end{equation}

On the other hand, the Dyson equation gives
\[
   G^{-1}
   =
   wI-a-\eta(G).
\]
Since
$
   wI-a
   =
   -a_x+i\varepsilon I
$
and
$
   \eta(G)
   =
   -\eta(u)-i\eta(\gamma),
$
we have
\begin{equation}\label{eq:dyson-parts}
   G^{-1}
   =
   -a_x+\eta(u)
   +
   i\bigl(
      \varepsilon I+\eta(\gamma)
   \bigr).
\end{equation}

Comparing the real parts of
\eqref{eq:normalizer-parts} and
\eqref{eq:dyson-parts} gives
\[
   bub+\eta(u)=a_x,
\]
which is \eqref{eq:dyson-real-unbalanced}.  Comparing imaginary parts gives
$
   b\gamma b
   =
   \varepsilon I+\eta(\gamma),
$
or equivalently
\[
   (\mathcal C_b-\eta)(\gamma)
   =
   \varepsilon I.
\]
This is \eqref{eq:dyson-imag-unbalanced}.
\end{proof}

\subsection{The balanced Dyson equations}

The preceding identities become especially transparent after passing to
balanced coordinates.

For $b\succ0$ recall
$
   \widehat\eta_b
   =
   \mathcal C_{b^{-1/2}}
   \eta
   \mathcal C_{b^{-1/2}}.
$
For a real spectral parameter $x$, set
\begin{equation}\label{eq:a-tilde-xb}
   \widetilde a_{x,b}
   :=
   b^{-1/2}a_xb^{-1/2}.
\end{equation}

Let
$
   U
   :=
   b^{1/2}Gb^{1/2}.
$
Since $b$ is a unitary normalizer, $U$ is unitary.  Write
\[
   U=P-iQ,
   \qquad
   P=P^*,
   \qquad
   Q=Q^*.
\]
If $\varepsilon>0$, then
$
   -\operatorname{Im}G(w)\succ0,
$
and hence
$
   Q\succ0.
$

\begin{corollary}[Balanced Dyson equations]
\label{cor:balanced-dyson}
With the notation above,
\begin{align}
   (I+\widehat\eta_b)P
   &=
   -\widetilde a_{x,b},
   \label{eq:balanced-real}\\
   (I-\widehat\eta_b)Q
   &=
   \varepsilon b^{-1}.
   \label{eq:balanced-imag}
\end{align}
\end{corollary}

\begin{proof}
Since
$
   G=-u-i\gamma
$
and
$
   U=b^{1/2}Gb^{1/2}=P-iQ,
$
we have
\begin{equation}\label{eq:PQ-ugamma}
   P
   =
   -b^{1/2}ub^{1/2},
   \qquad
   Q
   =
   b^{1/2}\gamma b^{1/2}.
\end{equation}

Apply the congruence map
$\mathcal C_{b^{-1/2}}$ to
\eqref{eq:dyson-real-unbalanced}.  Since
$
   u
   =
   -b^{-1/2}Pb^{-1/2},
$
we obtain
\[
   -P-\widehat\eta_b(P)
   =
   \widetilde a_{x,b}.
\]
This is \eqref{eq:balanced-real}.

Similarly, applying
$\mathcal C_{b^{-1/2}}$ to
\eqref{eq:dyson-imag-unbalanced} and using
$
   \gamma
   =
   b^{-1/2}Qb^{-1/2}
$
gives
\[
   Q-\widehat\eta_b(Q)
   =
   \varepsilon b^{-1},
\]
which is \eqref{eq:balanced-imag}.
\end{proof}

\begin{remark}[Two stability operators]
\label{rem:two-stability-operators}
The balanced Dyson equations explain the appearance of the two operators
\[
   I+\widehat\eta_b
   \qquad\text{and}\qquad
   I-\widehat\eta_b
\]
in the variational problem.

The real part of the Dyson solution is determined through
$I+\widehat\eta_b$, while its imaginary part is controlled by
$I-\widehat\eta_b$.  In particular, at a real boundary point
$\varepsilon=0$,
\[
   (I-\widehat\eta_b)Q=0.
\]
Thus a nonzero imaginary part of the boundary Dyson solution produces a
nontrivial kernel of the stability operator
$I-\widehat\eta_b$, and hence an active feasibility constraint for the limiting scale $b_z$.
\end{remark}

\section{Regularized variational problem}
\label{sec:regularized}

For every $w=z+i\eps$ with $\eps>0$, the physical Dyson solution
has a strictly negative definite imaginary part and admits a positive
unitary normalizer. The difficulty is that these normalizers need not
converge to a positive definite matrix as $\eps\downarrow0$.
We therefore first solve a variational problem at positive $\eps$,
leaving the real-axis limit to the next section.

The argument compares two constructions. At an arbitrary strictly
feasible scale $b$, quadratic maximization determines a pair
$(u_b,\gamma_b)$. Independently, the Dyson solution determines a
triple $(b_\eps,u_\eps,\gamma_\eps)$. We show that the two pairs
agree when $b=b_\eps$, and that the remaining normalization identity
is exactly the vanishing of the gradient of the reduced objective.

Recall that $a_z=a-z\Id$, $L_b^\pm=\Cop{b}\pm\eta$, and
$\Cop{b}(X)=bXb$. We use the convention
\begin{equation}\label{eq:regularized-K}
\begin{split}
 \KK_{z,\eps}(b,u,\gamma)
 &:={\KK}_z(b,u,\gamma)+\eps\trm\gamma\\
 &=\trm\log b+\ip{a_z}{u}
   -\frac12\ip{u}{L_b^+u}
   -\frac12\ip{\gamma}{L_b^-\gamma}
   +\eps\trm\gamma,
\end{split}
\end{equation}
where $b\succ0$, $u=u^*$, and $\gamma\succeq0$.
Define the reduced objective by
\begin{equation}\label{eq:regularized-F-definition}
 F_{z,\eps}(b)
 :=\sup_{\substack{u=u^*\\\gamma\succeq0}}
       \KK_{z,\eps}(b,u,\gamma),
 \qquad b\succ0.
\end{equation}

\subsection{Quadratic maximization at a fixed scale}

The following proposition does not use the Dyson equation. It describes
the auxiliary variables and the reduced objective at every positive scale.

\begin{proposition}[The regularized partial supremum]
\label{prop:elimination}
Fix $z\in\mathbb R$ and $\eps>0$. For $b\in\Fint$, the unique
maximizers in \eqref{eq:regularized-F-definition} are
\begin{equation}\label{eq:maximizers}
 u_b=(L_b^+)^{-1}a_z,
 \qquad
 \gamma_b=\eps(L_b^-)^{-1}\Id\succ0.
\end{equation}
Their value is
\begin{equation}\label{eq:partial-sup}
\boxed{\begin{aligned}
 F_{z,\eps}(b)
 ={}&\trm\log b
 +\frac12\ip{a_z}{(L_b^+)^{-1}a_z}\\
 &+\frac{\eps^2}{2}\ip{\Id}{(L_b^-)^{-1}\Id}.
\end{aligned}}
\end{equation}
For $b\notin\Fint$, the supremum is $+\infty$. Thus $\Fint$ is
exactly the effective domain of $F_{z,\eps}$.
\end{proposition}

\begin{proof}
For $b\in\Fint$, both $L_b^+$ and $L_b^-$ are positive definite
as Hilbert--Schmidt operators. Completing the square in the two
Hermitian variables gives the unconstrained maximizers
\eqref{eq:maximizers} and the value \eqref{eq:partial-sup}.
It remains to check that the second maximizer is positive definite.

Write
$\widehat\eta_b=\Cop{b^{-1/2}}\eta\Cop{b^{-1/2}}$.
The congruence identities give
\begin{equation}\label{eq:positive-inverse}
 (L_b^-)^{-1}
 =\Cop{b^{-1/2}}
   (\mathrm{id}-\widehat\eta_b)^{-1}
   \Cop{b^{-1/2}},
 \qquad
 (\mathrm{id}-\widehat\eta_b)^{-1}
 =\sum_{k=0}^{\infty}\widehat\eta_b^{\,k}.
\end{equation}
The series converges because $\|\widehat\eta_b\|_{2\to2}<1$.
Each summand is a positive map. On applying the full expression to
$\Id$, its zeroth summand is $b^{-2}$. Consequently,
\begin{equation*}
 (L_b^-)^{-1}\Id\succeq b^{-2}\succ0.
\end{equation*}
Hence $\gamma_b\succ0$: the positivity restriction on $\gamma$
does not change the maximizer or its value.

Now suppose that $b\notin\Fint$. By
Lemma~\ref{lem:feasible-pf}, the map $\widehat\eta_b$ has a
nonzero positive semidefinite eigenmatrix $P$ with
$\widehat\eta_b(P)=rP$, where
$r=\|\widehat\eta_b\|_{2\to2}\geq1$.
Put $\gamma_0=b^{-1/2}Pb^{-1/2}\succeq0$. Then
\begin{equation*}
 \ip{\gamma_0}{L_b^-\gamma_0}
 =(1-r)\|P\|_2^2,
 \qquad \trm\gamma_0>0.
\end{equation*}
Taking $u=0$ and $\gamma=t\gamma_0$ gives
\begin{equation*}
 \KK_{z,\eps}(b,0,t\gamma_0)
 =\trm\log b
   +\frac{t^2}{2}(r-1)\|P\|_2^2
   +\eps t\trm\gamma_0.
\end{equation*}
This tends to $+\infty$ as $t\to\infty$, whether $r>1$ or $r=1$.
\end{proof}

\subsection{The scale supplied by the Dyson solution}

Proposition~\ref{prop:elimination} starts with $b$ and constructs
$(u_b,\gamma_b)$. We now start instead with $G(z+i\eps)$.
The first two identities in the next lemma will identify its real
and imaginary parts with those same maximizing variables.
The third identity will make the gradient vanish.

\begin{lemma}[Unitary normalization and strict feasibility]
\label{lem:Dyson-scale}
Fix $w=z+i\eps$ with $\eps>0$, and write
\begin{equation*}
 G_\eps:=G(z+i\eps)=-u_\eps-i\gamma_\eps,
 \qquad \gamma_\eps\succ0.
\end{equation*}
Let $b_\eps\succ0$ be the unitary normalizer of $G_\eps$ given by
Lemma~\ref{lem:pup}.  It belongs to $\Fint$ and satisfies
\begin{equation}\label{eq:balanced}
\begin{aligned}
 L_{b_\eps}^+u_\eps&=a_z,\\
 L_{b_\eps}^-\gamma_\eps&=\eps\Id,\\
 b_\eps^{-1}&=u_\eps b_\eps u_\eps
                  +\gamma_\eps b_\eps\gamma_\eps.
\end{aligned}
\end{equation}
In particular, the maximizers from
Proposition~\ref{prop:elimination} at $b=b_\eps$ are exactly
\begin{equation}\label{eq:Dyson-auxiliary-identification}
 u_{b_\eps}=u_\eps,
 \qquad \gamma_{b_\eps}=\gamma_\eps.
\end{equation}
\end{lemma}

\begin{proof}
By Lemma~\ref{lem:pup}, applied with $X=-u_\eps$ and $Y=\gamma_\eps$,
the normalizer $b_\eps$ exists, and by \eqref{eq:normalizer-inverse},
$G_\eps^{-1}=b_\eps G_\eps^*b_\eps$.  The Dyson equation and this identity
give two expressions for the inverse: 
\begin{equation*}
\begin{aligned}
 G_\eps^{-1}
 &=-a_z+\eta(u_\eps)
        +i\bigl(\eps\Id+\eta(\gamma_\eps)\bigr),\\
 G_\eps^{-1}
 &=-b_\eps u_\eps b_\eps+i b_\eps\gamma_\eps b_\eps.
\end{aligned}
\end{equation*}
Comparing real and imaginary parts yields the first two equations in
\eqref{eq:balanced}. Adding
$G_\eps^*b_\eps G_\eps=b_\eps^{-1}$ and
$G_\eps b_\eps G_\eps^*=b_\eps^{-1}$ cancels the mixed terms
and gives the third equation.

To prove strict feasibility, put
$Q=b_\eps^{1/2}\gamma_\eps b_\eps^{1/2}\succ0$.
Conjugating the imaginary equation gives
\begin{equation*}
 Q-\widehat\eta_{b_\eps}(Q)=\eps b_\eps^{-1}.
\end{equation*}
Choose a nonzero $P_0\succeq0$ with
$\widehat\eta_{b_\eps}(P_0)=rP_0$ and
$r=\|\widehat\eta_{b_\eps}\|_{2\to2}$, as in
Lemma~\ref{lem:feasible-pf}. Hilbert--Schmidt self-adjointness gives
\begin{equation*}
\begin{aligned}
 \eps\ip{P_0}{b_\eps^{-1}}
 &=\ip{P_0}{Q}-\ip{P_0}{\widehat\eta_{b_\eps}(Q)}\\
 &=\ip{P_0}{Q}-\ip{\widehat\eta_{b_\eps}(P_0)}{Q}\\
 &=(1-r)\ip{P_0}{Q}.
\end{aligned}
\end{equation*}
Both pairings are strictly positive, since $P_0$ is nonzero and
positive semidefinite, whereas $Q$ and $b_\eps^{-1}$ are positive
definite. Therefore $r<1$, so $b_\eps\in\Fint$.
Now $L_{b_\eps}^\pm$ are invertible. The first two equations of
\eqref{eq:balanced} uniquely determine $u_\eps$ and $\gamma_\eps$
and coincide with \eqref{eq:maximizers}, proving
\eqref{eq:Dyson-auxiliary-identification}.
\end{proof}

\subsection{First variation and convexity}

We next differentiate the \emph{regularized} reduced objective.
In contrast to the first variation of $F_z$, the derivative contains
one contribution from each maximizing variable.

\begin{proposition}[First variation of the regularized objective]
\label{prop:regularized-first-variation}
For $b\in\Fint$, let $u_b$ and $\gamma_b$ be the maximizers in
\eqref{eq:maximizers}. For every Hermitian variation $h$,
\begin{equation}\label{eq:regularized-first-variation}
 DF_{z,\eps}(b)[h]
 =\trm\!\left[
    \bigl(b^{-1}-u_bbu_b-\gamma_bb\gamma_b\bigr)h
   \right].
\end{equation}
Thus, with respect to the normalized Hilbert--Schmidt pairing,
\begin{equation}\label{eq:Freg-gradient}
 \boxed{
 \nabla F_{z,\eps}(b)
 =b^{-1}-u_bbu_b-\gamma_bb\gamma_b.
 }
\end{equation}
\end{proposition}

\begin{proof}
Both shifted operators have the same differential:
\begin{equation*}
 D L_b^\pm[h](Y)=D\Cop{b}[h](Y)=hYb+bYh.
\end{equation*}
Using
$D(L_b^\pm)^{-1}[h]=-(L_b^\pm)^{-1}D\Cop{b}[h](L_b^\pm)^{-1}$
in \eqref{eq:partial-sup}, and the definitions of $u_b$ and
$\gamma_b$, gives
\begin{equation*}
\begin{split}
 DF_{z,\eps}(b)[h]
 ={}&\trm(b^{-1}h)
   -\frac12\ip{u_b}{hu_bb+bu_bh}\\
   &-\frac12\ip{\gamma_b}{h\gamma_bb+b\gamma_bh}.
\end{split}
\end{equation*}
For every Hermitian $v$, cyclicity of the trace gives
$\frac12\ip{v}{hvb+bvh}=\trm(vbv h)$.
Apply this once with $v=u_b$ and once with $v=\gamma_b$ to obtain
\eqref{eq:regularized-first-variation}.
\end{proof}

\begin{corollary}[Geodesic convexity]
\label{cor:Freg-convex}
For every $z\in\mathbb R$ and $\eps>0$, $F_{z,\eps}$ is
geodesically convex on the geodesically convex set $\Fint$.
\end{corollary}

\begin{proof}
By Corollary~\ref{cor:two-inverse-geodesic}, both
$b\mapsto(L_b^+)^{-1}$ and $b\mapsto(L_b^-)^{-1}$ are convex
in the Hilbert--Schmidt operator order along geodesics in $\Fint$.
Pairing them with the fixed matrices $a_z$ and $\Id$, respectively,
proves convexity of the two quadratic terms in \eqref{eq:partial-sup}.
The remaining term is affine along
$b_t=b_0^{1/2}e^{tX}b_0^{1/2}$:
\begin{equation*}
 \trm\log b_t=\trm\log b_0+t\trm X.
\end{equation*}
\end{proof}

\subsection{Regularized variational identity}

\begin{theorem}[Regularized variational identity]
\label{thm:regularized}
For every $z\in\mathbb R$ and $\eps>0$,
\begin{equation}\label{eq:regularized-theorem}
 \boxed{
 \log\Delta(S_a-z-i\eps)
 =-\frac12+\min_{b\in\Fint}F_{z,\eps}(b).
 }
\end{equation}
The minimum is attained at the scale $b_\eps$ from
Lemma~\ref{lem:Dyson-scale}. Equivalently,
\begin{equation}\label{eq:full-regularized}
 \log\Delta(S_a-z-i\eps)
 =-\frac12+\min_{b\succ0}
    \sup_{\substack{u=u^*\\\gamma\succeq0}}
       \KK_{z,\eps}(b,u,\gamma).
\end{equation}
\end{theorem}

\begin{proof}
\emph{Identification of the maximizing variables and stationarity.}
At the Dyson scale, the first two balanced equations give explicitly
\begin{equation*}
 u_{b_\eps}=(L_{b_\eps}^+)^{-1}a_z=u_\eps,
 \qquad
 \gamma_{b_\eps}=\eps(L_{b_\eps}^-)^{-1}\Id=\gamma_\eps.
\end{equation*}
Substituting these identities into the gradient formula yields
\begin{equation*}
\begin{aligned}
 \nabla F_{z,\eps}(b_\eps)
 &=b_\eps^{-1}-u_\eps b_\eps u_\eps
                    -\gamma_\eps b_\eps\gamma_\eps\\
 &=\underbrace{\bigl(b_\eps^{-1}-u_\eps b_\eps u_\eps\bigr)}
      _{\displaystyle=\,\gamma_\eps b_\eps\gamma_\eps}
      -\gamma_\eps b_\eps\gamma_\eps
 =0.
\end{aligned}
\end{equation*}
Here the underbraced equality is exactly the third balanced equation.
Thus the first two equations identify the inner maximizers, while the
third makes their reduced objective stationary in $b$.

\emph{Global minimality.}
Fix $c\in\Fint$ and take the geodesic
\begin{equation*}
 b_t=b_\eps^{1/2}e^{tX}b_\eps^{1/2},
 \qquad
 X=\log(b_\eps^{-1/2}cb_\eps^{-1/2}),\quad 0\leq t\leq1.
\end{equation*}
It lies in $\Fint$. The function $f(t)=F_{z,\eps}(b_t)$ is convex
by Corollary~\ref{cor:Freg-convex}, and $f'(0)=0$ by stationarity.
Therefore $F_{z,\eps}(c)=f(1)\geq f(0)=F_{z,\eps}(b_\eps)$.
This proves global minimality without requiring compactness of a
minimizing sequence.

\emph{Value at the minimum.}
Write $G=G(z+i\eps)=-u_\eps-i\gamma_\eps$.
At the identified maximizers, the explicit formula for the objective
becomes
\begin{equation*}
 F_{z,\eps}(b_\eps)
 =\trm\log b_\eps
   +\frac12\ip{a_z}{u_\eps}
   +\frac{\eps}{2}\trm\gamma_\eps.
\end{equation*}
Multiplying the Dyson equation by $G$, taking the normalized trace,
and then taking real parts gives
\begin{equation*}
 \ip{a_z}{u_\eps}+\eps\trm\gamma_\eps
 =1+\operatorname{Re}\trm\bigl(G\eta(G)\bigr).
\end{equation*}
Also, unitarity of $b_\eps^{1/2}Gb_\eps^{1/2}$ implies
$\trm\log b_\eps=-m^{-1}\log|\det G|$. Consequently,
\begin{equation*}
\begin{aligned}
 F_{z,\eps}(b_\eps)-\frac12
 &=-\frac1m\log|\det G|
    +\frac12\operatorname{Re}\trm\bigl(G\eta(G)\bigr)\\
 &=\log\Delta(S_a-z-i\eps),
\end{aligned}
\end{equation*}
where the last equality is the closed determinant identity of
Theorem~\ref{thm:closed}. This proves
\eqref{eq:regularized-theorem}.
Finally, $F_{z,\eps}=\sup_{u,\gamma}\KK_{z,\eps}$ by definition,
and Proposition~\ref{prop:elimination} makes this supremum infinite
off $\Fint$. This proves \eqref{eq:full-regularized}.
\end{proof}

\begin{remark}[Relation with the formula of \cite{Yin2026}]
\label{rem:yin-dictionary}
Let $a_0=a_z$ and let $G_\eps=-u_\eps-i\gamma_\eps$ be as in
Lemma~\ref{lem:Dyson-scale}.  At the regularized optimum, the trial matrix
of \cite{Yin2026} is $\gamma_\eps/\eps$ and $Z=u_\eps$, whereas
our scale satisfies
$b_\eps\gamma_\eps b_\eps=\eps\Id+\eta(\gamma_\eps)$ by
\eqref{eq:balanced}.  Thus the two regularized objectives take the same value at their
respective minimizers.  The correspondence
$b_\eps\leftrightarrow\gamma_\eps/\eps$ is a special case of the
bijection described in Section~\ref{sec:further-questions}.
\end{remark}

\section{Real-axis variational identity }
\label{sec:real-axis}

The passage to the real axis uses a simple property of one-sided regularization. The added term is nonnegative and tends to zero at each fixed strictly feasible competitor. Consequently, the regularized infimum cannot fall below the unregularized infimum; conversely, any fixed nearly optimal competitor gives an asymptotically matching upper bound. Thus the optimal values converge without any compactness assumption on the minimizing scales. Radial approximation then allows boundary competitors.

\begin{theorem}[Real-axis variational identity]
\label{thm:real-axis}
For every $z\in\R$,
\begin{equation}\label{eq:real-theorem}
 \boxed{
 \log\Delta(S_a-z)
 =-\frac12+\inf_{b\in\Fset}F_z(b)
 }
\end{equation}
in $[-\infty,\infty)$. 
\end{theorem}

\begin{proof}
For each fixed $b\in\Fint$, positivity of $(L_b^-)^{-1}$ as a
Hilbert--Schmidt operator gives
\begin{equation}\label{eq:pointwise-limit}
 \Freg_{z,\eps}(b)
 =F_z(b)+\frac{\eps^2}{2}\ip{\Id}{(L_b^-)^{-1}\Id}
 \downarrow F_z(b)
 \quad(\eps\downarrow0).
\end{equation}
Put $m_\eps=\inf_{\Fint}\Freg_{z,\eps}$ and
$m_0=\inf_{\Fint}F_z$. These infima are bounded above by the value
at a fixed strictly feasible matrix, for instance a sufficiently large
multiple of $\Id$. The numbers $m_\eps$ decrease as
$\eps\downarrow0$ and satisfy $m_\eps\geq m_0$. Conversely, for
every fixed $b\in\Fint$,
\begin{equation*}
 \lim_{\eps\downarrow0}m_\eps
 \leq\lim_{\eps\downarrow0}\Freg_{z,\eps}(b)=F_z(b).
\end{equation*}
Taking the infimum over $b$ proves
\begin{equation}\label{eq:inf-limit}
 \lim_{\eps\downarrow0}m_\eps=m_0,
\end{equation}
including when $m_0=-\infty$. This step uses fixed competitors, not
compactness of minimizing sequences.

On the spectral side,
\begin{equation*}
 \log\Delta(S_a-z-i\eps)
 =\frac12\int_\R\log\bigl((t-z)^2+\eps^2\bigr)\,d\mu(t)
 \downarrow\log\Delta(S_a-z).
\end{equation*}
To justify the limit even when it is $-\infty$, subtract the logarithm
from a fixed upper bound valid for $0<\eps\leq1$ on the compact
support of $\mu$, and apply monotone convergence.

Passing to the limit in Theorem~\ref{thm:regularized}, using
\eqref{eq:inf-limit} and Lemma~\ref{lem:radial-approximation}, proves
\eqref{eq:real-theorem}.
\end{proof}

\begin{corollary}[Characterization of the value $-\infty$]
\label{cor:minus-infinity-atoms}
For every $z\in\mathbb R$,
\begin{equation*}
   \log\Delta(S_a-z)=-\infty
   \quad\Longleftrightarrow\quad
   \inf_{b\in\mathcal F}F_z(b)=-\infty
   \quad\Longleftrightarrow\quad
   \mu_{S_a}(\{z\})>0.
\end{equation*}
In particular, both quantities in the real-axis variational
identity are finite at every non-atomic point.
\end{corollary}

\begin{proof}
By Theorem~\ref{thm:real-axis},
\begin{equation*}
   \inf_{b\in\mathcal F}F_z(b)
   =
   \log\Delta(S_a-z)+\frac12.
\end{equation*}
Thus it suffices to characterize when
\begin{equation*}
   \log\Delta(S_a-z)
   =
   \int_{\mathbb R}\log|t-z|\,d\mu_{S_a}(t)
\end{equation*}
equals $-\infty$. Write $\mu:=\mu_{S_a}$.

Suppose first that $\mu(\{z\})>0$.
The integrand equals $-\infty$ on a set of positive
$\mu$-measure. Its positive part is integrable because
$\mu$ has compact support. Hence
$\log\Delta(S_a-z)=-\infty$.

Conversely, suppose that $\mu(\{z\})=0$.
By Theorem~\ref{thm:nonatomic-local-regularity}, there exist
$0<r<1$, an exponent $1<p<\infty$, and a nonnegative
function $\rho\in L^p(J)$, where $J=(z-r,z+r)$, such that
\begin{equation*}
   d\mu(t)\big|_J=\rho(t)\,dt.
\end{equation*}
Let $q=p/(p-1)$. H\"older's inequality gives
\begin{align*}
   \int_J|\log|t-z||\,d\mu(t)
   &=
   \int_J|\log|t-z||\,\rho(t)\,dt\\
   &\leq
   \|\rho\|_{L^p(J)}
   \left(
      \int_{-r}^{r}|\log|s||^q\,ds
   \right)^{1/q}
   <\infty.
\end{align*}
The last integral is finite since
\begin{equation*}
   \int_{-r}^{r}|\log|s||^q\,ds
   =
   2\int_{\log(1/r)}^\infty y^q e^{-y}\,dy
   <\infty.
\end{equation*}
On $\operatorname{supp}\mu\setminus J$, the function
$t\mapsto|\log|t-z||$ is bounded, because the support is
compact and $|t-z|\geq r$. Therefore
\begin{equation*}
   \int_{\mathbb R}|\log|t-z||\,d\mu(t)<\infty,
\end{equation*}
so $\log\Delta(S_a-z)$ is finite. The variational identity
then gives finiteness of $\inf_{b\in\mathcal F}F_z(b)$
as well.
\end{proof}

\section{Attainment}
\label{sec:attainment}
\begin{proposition}[Attainment at a positive definite cluster point]
\label{prop:attainment-cluster-point}
Fix $z\in\mathbb R$, and let $b_\varepsilon$ be the
regularized Dyson scales. Suppose that, for some sequence
$\varepsilon_k\downarrow0$,
$
   b_{\varepsilon_k}\longrightarrow b_*\succ0.
$
Then $b_*\in\mathcal F$, the infimum
$
   m_z:=\inf_{b\in\mathcal F}F_z(b)
$
is finite, and
$
   F_z(b_*)=m_z.
$

In particular, if $m_z=-\infty$, the regularized scales
have no positive definite cluster point as
$\varepsilon\downarrow0$.
\end{proposition}

\begin{proof}
We first recall why the regularized scales form a minimizing
family for the unregularized objective. Put
\begin{equation*}
   m_\varepsilon
   :=
   \min_{b\in\operatorname{int}\mathcal F}
      F_{z,\varepsilon}(b).
\end{equation*}
The real-axis limiting argument gives
$m_\varepsilon\to m_z$. Since the regularization term is
nonnegative and $b_\varepsilon$ minimizes
$F_{z,\varepsilon}$,
\begin{equation*}
   m_z
   \leq F_z(b_\varepsilon)
   \leq F_{z,\varepsilon}(b_\varepsilon)
   =m_\varepsilon.
\end{equation*}
It follows that
\begin{equation}\label{eq:attainment-minimizing-family}
   F_z(b_\varepsilon)\longrightarrow m_z
   \qquad(\varepsilon\downarrow0),
\end{equation}
also when $m_z=-\infty$.

Next, $b_{\varepsilon_k}\in\mathcal F$ and continuity of
$b\mapsto\mathcal C_b$ give
\begin{equation*}
   \mathcal C_{b_*}-\eta
   =
   \lim_{k\to\infty}
   \bigl(\mathcal C_{b_{\varepsilon_k}}-\eta\bigr)
   \succeq0.
\end{equation*}
Because $b_*\succ0$, this proves $b_*\in\mathcal F$.

We use the supremum definition to pass to the limit in the
objective, without assuming that $L_{b_*}^+$ is invertible.
For every fixed Hermitian matrix $u$, set
\begin{equation*}
   \ell_z(b,u)
   :=
   \operatorname{tr}_m\log b
   +\langle a_z,u\rangle
   -\frac12\langle u,L_b^+u\rangle.
\end{equation*}
The function $b\mapsto\ell_z(b,u)$ is continuous on the
positive definite cone, and
\begin{equation*}
   F_z(b)=\sup_{u=u^*}\ell_z(b,u),
   \qquad b\in\mathcal F.
\end{equation*}
Therefore, for every fixed $u=u^*$,
\begin{equation*}
   \ell_z(b_*,u)
   =
   \lim_{k\to\infty}\ell_z(b_{\varepsilon_k},u)
   \leq
   \liminf_{k\to\infty}F_z(b_{\varepsilon_k}).
\end{equation*}
Taking the supremum over $u$ proves lower semicontinuity
along this sequence:
\begin{equation*}
   F_z(b_*)
   \leq
   \liminf_{k\to\infty}F_z(b_{\varepsilon_k})
   =m_z.
\end{equation*}

Taking $u=0$ in the supremum definition gives
\begin{equation*}
   F_z(b_*)\geq\operatorname{tr}_m\log b_*>-\infty.
\end{equation*}
Hence $m_z=-\infty$ is impossible under the cluster-point
hypothesis. Moreover, $m_z<+\infty$, since $F_z$ is finite
at any strictly feasible matrix. Thus $m_z$ is finite.
Finally, feasibility of $b_*$ gives the reverse inequality
$m_z\leq F_z(b_*)$, and therefore
$
   F_z(b_*)=m_z.
$
\end{proof}

\begin{proposition}[Limit of the regularized normalizers]
\label{prop:normalizer-limit}
Suppose that
$
   G:=G(z+i0)
$
exists.

\begin{enumerate}
\item[(i)]
If
$
   -\operatorname{Im}G\succ0,
$
then
$
   b_\varepsilon\longrightarrow b_z,
$
where $b_z$ is the unitary normalizer obtained by applying the $pUp$
formula of Lemma~\ref{lem:pup} directly to $G$.

\item[(ii)]
Suppose that $z$ belongs to an open component of
$\mathbb R\setminus\supp\mu_{S_a}$.  Then $G(z)=G(z)^*$, and
$
   b_\varepsilon\longrightarrow b_z,
$
where
\begin{equation}\label{eq:gap-normalizer}
   b_z^{-1}
   =
   D^{1/2}
   \left|
      D^{-1/2}G(z)D^{-1/2}
   \right|
   D^{1/2},
   \qquad
   D:=-G'(z)\succ0.
\end{equation}
The matrix $b_z$ is a unitary normalizer of $G(z)$.  If $G(z)$ is
definite, in particular on either unbounded component of the resolvent
set, then
\begin{equation}\label{eq:definite-normalizer}
   b_z=|G(z)|^{-1}.
\end{equation}
\end{enumerate}
\end{proposition}

The proof is given in Appendix~\ref{app:normalizer-limits}.

\begin{corollary}[Attainment at nondegenerate regular points]
\label{cor:regularized-value-limit}
Fix $z\in\mathbb R$. Suppose that the finite boundary value
\begin{equation*}
   G:=G(z+i0)
   =
   \lim_{\varepsilon\downarrow0}G(z+i\varepsilon)
\end{equation*}
exists, and that either $-\operatorname{Im}G\succ0$ or $z$
belongs to an open component of
$\mathbb R\setminus\operatorname{supp}\mu_{S_a}$.

Let $b_z$ be the limiting scale from
Proposition~\ref{prop:normalizer-limit}.
Then $b_z\in\mathcal F$, the infimum of $F_z$ on
$\mathcal F$ is finite and attained at $b_z$, and
\begin{equation}\label{eq:value-limiting-scale}
   F_z(b_z)
   =
   \lim_{\varepsilon\downarrow0}F_z(b_\varepsilon)
   =
   \log\Delta(S_a-z)+\frac12.
\end{equation}
\end{corollary}

\begin{proof}
By Proposition~\ref{prop:normalizer-limit},
$b_\varepsilon\to b_z\succ0$. Hence
Proposition~\ref{prop:attainment-cluster-point} gives
$b_z\in\mathcal F$ and
\begin{equation*}
   F_z(b_z)=\inf_{b\in\mathcal F}F_z(b)\in\mathbb R.
\end{equation*}
The regularized scales form a minimizing family for $F_z$,
so $F_z(b_\varepsilon)$ converges to this infimum.
Finally, Theorem~\ref{thm:real-axis} identifies the infimum
with $\log\Delta(S_a-z)+\frac12$, proving
\eqref{eq:value-limiting-scale}.
\end{proof}

The balanced Dyson equations also locate the minimizer relative to the
boundary of the feasible set.

\begin{proposition}[Location of the minimizer]
\label{prop:active-set}
Fix $z\in\mathbb R$ and let $b_\varepsilon$ be the regularized scales.
\begin{enumerate}
\item[(i)]
Suppose that $G:=G(z+i0)$ exists, that
$\operatorname{Im}G\neq0$, and that
$b_{\varepsilon_k}\to b_*\succ0$ for some sequence
$\varepsilon_k\downarrow0$.  Then $b_*\in\partial\mathcal F$.  In
particular, if $-\operatorname{Im}G\succ0$, then
$b_z\in\partial\mathcal F$.

\item[(ii)]
Suppose that $z$ belongs to an open component of
$\mathbb R\setminus\supp\mu_{S_a}$.  Then
$b_z\in\operatorname{int}\mathcal F$, and
\begin{equation}\label{eq:gap-stability}
   \bigl(\mathcal C_{b_z}-\eta\bigr)(D)=I,
   \qquad
   D=-G'(z)\succ0.
\end{equation}
\end{enumerate}
Consequently, at a nondegenerate regular point $z$,
\begin{equation*}
   b_z\in\partial\mathcal F
   \quad\Longleftrightarrow\quad
   z\in\supp\mu_{S_a}.
\end{equation*}
\end{proposition}

\begin{proof}
Write $G(z+i\varepsilon)=-u_\varepsilon-i\gamma_\varepsilon$ with
$\gamma_\varepsilon\succ0$.  By Lemma~\ref{lem:Dyson-scale},
\begin{equation}\label{eq:active-set-imag}
   \bigl(\mathcal C_{b_\varepsilon}-\eta\bigr)(\gamma_\varepsilon)
   =\varepsilon I
   \qquad(\varepsilon>0).
\end{equation}

\smallskip
\noindent\emph{Proof of (i).}
Put $\gamma:=-\operatorname{Im}G\succeq0$, so that
$\gamma_{\varepsilon_k}\to\gamma\neq0$.  Passing to the limit in
\eqref{eq:active-set-imag} along $\varepsilon_k$ gives
$
   \bigl(\mathcal C_{b_*}-\eta\bigr)(\gamma)=0.
$
Applying $\mathcal C_{b_*^{-1/2}}$ and writing
$Q:=b_*^{1/2}\gamma\,b_*^{1/2}\neq0$, we obtain
$\widehat\eta_{b_*}(Q)=Q$.  Hence $\|\widehat\eta_{b_*}\|\geq1$.  On the
other hand, $b_*\in\mathcal F$ by
Proposition~\ref{prop:attainment-cluster-point}, so
$\|\widehat\eta_{b_*}\|\leq1$ by Lemma~\ref{lem:feasible-pf}.  Thus
$\|\widehat\eta_{b_*}\|=1$, and \eqref{eq:interior-feasible} shows that
$b_*$ is not an interior point of $\mathcal F$.  If
$-\operatorname{Im}G\succ0$, then $b_\varepsilon\to b_z$ by
Proposition~\ref{prop:normalizer-limit}(i), and the first statement
applies with $b_*=b_z$.

\smallskip
\noindent\emph{Proof of (ii).}
In a gap, $b_\varepsilon\to b_z\succ0$ by
Proposition~\ref{prop:normalizer-limit}(ii), and by
\eqref{eq:app-Y-expansion},
\begin{equation*}
   \varepsilon^{-1}\gamma_\varepsilon\longrightarrow D=-G'(z)\succ0.
\end{equation*}
Dividing \eqref{eq:active-set-imag} by $\varepsilon$ and letting
$\varepsilon\downarrow0$ gives \eqref{eq:gap-stability}.  Apply
$\mathcal C_{b_z^{-1/2}}$ and put
$Q':=b_z^{1/2}D\,b_z^{1/2}\succ0$:
\begin{equation*}
   Q'-\widehat\eta_{b_z}(Q')=b_z^{-1}.
\end{equation*}
Now argue as in the proof of Lemma~\ref{lem:Dyson-scale}.  Choose a
nonzero $P_0\succeq0$ with $\widehat\eta_{b_z}(P_0)=rP_0$ and
$r=\|\widehat\eta_{b_z}\|$, as in Lemma~\ref{lem:feasible-pf}.
Hilbert--Schmidt self-adjointness of $\widehat\eta_{b_z}$ gives
\begin{equation*}
   (1-r)\langle P_0,Q'\rangle
   =\langle P_0,Q'\rangle-\langle\widehat\eta_{b_z}(P_0),Q'\rangle
   =\langle P_0,b_z^{-1}\rangle.
\end{equation*}
Both pairings are strictly positive, because $P_0$ is nonzero and
positive semidefinite while $Q'$ and $b_z^{-1}$ are positive definite.
Therefore $r<1$, and $b_z\in\operatorname{int}\mathcal F$ by
\eqref{eq:interior-feasible}.

\smallskip
\noindent\emph{The equivalence.}
At a nondegenerate regular point either
$-\operatorname{Im}G(z+i0)\succ0$, in which case $z\in\supp\mu_{S_a}$
and $b_z\in\partial\mathcal F$ by (i), or $z$ lies in a gap, in which
case $b_z\in\operatorname{int}\mathcal F$ by (ii).
\end{proof}

\begin{remark}
\label{rem:active-set}
Part~(i) requires only $\operatorname{Im}G(z+i0)\neq0$.  Hence at a
degenerate regular point every positive definite cluster point of the
regularized scales lies on $\partial\mathcal F$.  The identity
\eqref{eq:gap-stability} is the first-order term in $\varepsilon$ of
the imaginary Dyson equation \eqref{eq:dyson-imag-unbalanced}; in a gap
the zeroth-order term vanishes identically, and it is the derivative
$D=-G'(z)$ that certifies strict feasibility.
\end{remark}

We record a general fact.

\begin{lemma}[Boundary values are invertible]
\label{lem:boundary-invertible}
Let $T$ be a self-adjoint matrix semicircle with mean $a'$ and covariance
map $\kappa$, and let $y\in\mathbb R$ be a point at which the boundary
value $G_T(y+i0)$ exists.  Then $G_T(y+i0)$ is invertible.
\end{lemma}

\begin{proof}
By the Dyson equation,
$G_T(w)^{-1}=wI-a'-\kappa\bigl(G_T(w)\bigr)$, and the right-hand side
converges as $w\to y$.  Hence $G_T(w)^{-1}$ converges, and passing to the
limit in $G_T(w)G_T(w)^{-1}=I$ shows that $G_T(y+i0)$ is invertible.
\end{proof}

Call $x\in\mathbb R$ a \emph{degenerate regular point} if $G(x+i0)$
exists and $-\operatorname{Im}G(x+i0)$ is nonzero but singular.  

\begin{remark}[Attainment and the behavior of the regularized scales]
\label{rem:attainment}
Corollary~\ref{cor:regularized-value-limit} proves attainment
at nondegenerate regular points. More generally,
Proposition~\ref{prop:attainment-cluster-point} shows that
any positive definite cluster point of the regularized
scales is a minimizer. We explain what this criterion gives
in the different boundary regimes.

\smallskip
\noindent\emph{Finite boundary values.}
Suppose that the finite limit $G:=G(x+i0)$ exists.
By Lemma~\ref{lem:boundary-invertible}, $G$ is invertible.
The unitary factorization
\begin{equation*}
   G(x+i\varepsilon)
   =
   b_\varepsilon^{-1/2}
   U_\varepsilon
   b_\varepsilon^{-1/2}
\end{equation*}
therefore gives
\begin{equation*}
   \det b_\varepsilon
   =
   |\det G(x+i\varepsilon)|^{-1}
   \longrightarrow
   |\det G|^{-1}\in(0,\infty).
\end{equation*}

Consequently, an upper bound on the scales along a sequence
also gives a uniform positive lower bound on their eigenvalues.
Indeed, suppose that
$\|b_{\varepsilon_k}\|\leq M$ for some
$\varepsilon_k\downarrow0$. For all sufficiently large $k$,
we have $\det b_{\varepsilon_k}\geq d$ for some $d>0$, and hence
\begin{equation*}
   \lambda_{\min}(b_{\varepsilon_k})
   \geq
   \frac{\det b_{\varepsilon_k}}
        {\|b_{\varepsilon_k}\|^{m-1}}
   \geq
   \frac{d}{M^{m-1}}>0.
\end{equation*}
Thus these scales lie in a compact subset of the positive
definite cone. A further subsequence converges to
$b_*\succ0$, and
Proposition~\ref{prop:attainment-cluster-point} gives
\begin{equation*}
   b_*\in\mathcal F,
   \qquad
   F_x(b_*)
   =
   \inf_{b\in\mathcal F}F_x(b)
   =
   \log\Delta(S_a-x)+\frac12.
\end{equation*}

Therefore, whenever $G(x+i0)$ exists as a finite matrix,
boundedness of the regularized scales along just one sequence
is sufficient for attainment. Convergence of the whole family
is not required. We do not know whether such a bounded sequence
always exists, in particular at degenerate regular points
and spectral edges. In Example~\ref{exa:coupled}, numerical
computations suggest convergence of the normalizers.

\smallskip
\noindent\emph{Singular boundary values.}
If $G(x+i0)$ does not exist as a finite matrix, the Puiseux
expansion in Lemma~\ref{lem:local-puiseux-control} gives
$
   \|G(x+i\varepsilon)\|\longrightarrow\infty.
$
The unitary factorization implies
\begin{equation*}
   \|G(x+i\varepsilon)\|
   \leq
   \|b_\varepsilon^{-1}\|,
\end{equation*}
and therefore
\begin{equation*}
   \lambda_{\min}(b_\varepsilon)
   =
   \|b_\varepsilon^{-1}\|^{-1}
   \leq
   \|G(x+i\varepsilon)\|^{-1}
   \longrightarrow0.
\end{equation*}
Thus the regularized scales leave every compact subset of the
positive definite cone, and the cluster-point criterion does
not apply. This escape does not by itself rule out attainment
by some other feasible matrix: the criterion is sufficient,
not an established characterization of attainment.

\smallskip
\noindent\emph{Non-attainment.}
At an atom, the infimum is $-\infty$ and cannot be attained,
because every $b\in\mathcal F$ satisfies
\begin{equation*}
   F_x(b)
   \geq
   \operatorname{tr}_m\log b
   >
   -\infty.
\end{equation*}
At general non-atomic singular points, the preceding
compactness argument leaves attainment unresolved.

The centered case provides a more precise conclusion.
For $a=0$ and $x=0$,
Theorem~\ref{thm:centered-capacity} states that the infimum
is attained if and only if $\eta$ is symmetrically
doubly-stochastic scalable. In particular, when
$\operatorname{Cap}(\eta)>0$ but $\eta$ is not symmetrically
doubly-stochastic scalable, the infimum is finite but is
not attained.
\end{remark}

\begin{exa}[Degenerate points filling the support]
\label{exa:coupled}
Let $m=2$, $r=1$, $A_1=E_{11}$ and
$a=\left(\begin{smallmatrix}0&c\\c&d\end{smallmatrix}\right)$ with $c\neq0$,
so that
\[
   S_a=\begin{pmatrix}s&c\\ c&d\end{pmatrix}.
\]
The covariance pencil is reducible, but the mean couples the two
coordinates.  All entries commute, so
$S_a=\lambda_+(s)P_+(s)+\lambda_-(s)P_-(s)$ with
$\lambda_\pm(t)=\tfrac12\bigl(t+d\pm\sqrt{(t-d)^2+4c^2}\bigr)$ and rank-one
spectral projections $P_\pm(t)$.  Since $\lambda_\pm'>0$ and
$\lambda_+>d>\lambda_-$, the density of the matrix-valued spectral measure
at $x=\lambda_\pm(t)$ is
$P_\pm(t)\,\rho_{\mathrm{sc}}(t)/\lambda_\pm'(t)$.  It is bounded and has
rank exactly one at every interior point of the support
$\lambda_+([-2,2])\cup\lambda_-([-2,2])$.  Hence every interior point of
the support is a degenerate regular point, and none of them is a limit of
nondegenerate regular points.  By Theorem~\ref{thm:real-axis} the
variational formula nevertheless holds there, with
\[
   \log\Delta(S_a-x)
   =
   \tfrac12\int_{-2}^2\log\bigl|(t-x)(d-x)-c^2\bigr|\,d\mu_{\mathrm{sc}}(t).
\]
\end{exa}

\bigskip

\section{Brown measure}
\label{sec:brown}
Theorem~\ref{thm:real-axis} is stated for self-adjoint matrix
semicircles and real spectral parameters.  Hermitization removes both
restrictions.

\begin{corollary}[Brown measure of matrix-valued semicircular elements]
\label{cor:brown}
Let $a_0,a_1,\dots,a_r\in M_m(\mathbb C)$ be arbitrary matrices and
\[
   S:=a_0\otimes1+\sum_{i=1}^ra_i\otimes s_i .
\]
For $z\in\mathbb C$ put
\[
   A_0(z):=\begin{pmatrix}0&a_0-zI\\ a_0^*-\bar zI&0\end{pmatrix},
   \qquad
   A_i:=\begin{pmatrix}0&a_i\\ a_i^*&0\end{pmatrix}
   \quad(i=1,\dots,r),
\]
all Hermitian in $M_{2m}(\mathbb C)$, and let
$\eta^{\mathrm h}(B):=\sum_iA_iBA_i$, $L^{\mathrm h}_b:=\mathcal C_b+\eta^{\mathrm h}$,
\begin{align*}
   \Fset^{\mathrm h}&:=\{b\succ0:\ \mathcal C_b-\eta^{\mathrm h}\succeq0\}
   \subset M_{2m}(\mathbb C), 
   \\
   F^{\mathrm h}_z(b)
   &:=
   \tr_{2m}\log b
   +
   \sup_{u=u^*}
   \Bigl\{
      \ip{A_0(z)}{u}-\frac12\ip{u}{L^{\mathrm h}_bu}
   \Bigr\}.
\end{align*}
Then, for every $z\in\mathbb C$,
\begin{equation}\label{eq:brown-variational}
   \log\Delta(S-z)
   =
   -\frac12+\inf_{b\in\Fset^{\mathrm h}}F^{\mathrm h}_z(b)
   \qquad\text{in }[-\infty,\infty),
\end{equation}
and both sides equal $-\infty$ if and only if $S-z$ has a nontrivial
kernel.  If $S-z$ is invertible, the infimum is attained.  Consequently the
Brown measure of $S$ is the distributional Laplacian
$\frac1{2\pi}\nabla_z^2$ of the right-hand side of
\eqref{eq:brown-variational}.
\end{corollary}

\begin{proof}
Let
\[
   H_z
   :=
   \begin{pmatrix}0&S-z\\ S^*-\bar z&0\end{pmatrix}
   =
   A_0(z)\otimes1+\sum_{i=1}^rA_i\otimes s_i
   \in M_{2m}(\mathbb C)\otimes\mathcal A .
\]
This is a self-adjoint matrix semicircle with Hermitian mean $A_0(z)$ and
Hermitian Kraus operators $A_i$, so Theorem~\ref{thm:real-axis} applies
to it at the real point $x=0$.  Its feasible set and functional are
$\Fset^{\mathrm h}$ and $F^{\mathrm h}_z$, because the shift $A_0(z)-0\cdot I$
is $A_0(z)$.  Since $S^*-\bar z=(S-z)^*$, the properties of the
Fuglede--Kadison determinant \cite{FugledeKadison1952} give
$\Delta(H_z)=\Delta(S-z)^{1/2}\Delta((S-z)^*)^{1/2}=\Delta(S-z)$; see
\cite[Lemma~3.5(ii)]{MaiSpeicher2024}.  This proves
\eqref{eq:brown-variational}.

By Corollary~\ref{cor:minus-infinity-atoms}, both sides are $-\infty$ exactly when
$\mu_{H_z}$ has an atom at $0$, that is, when $H_z$ has a nontrivial
kernel.  Since $\ker H_z=\ker(S-z)\oplus\ker((S-z)^*)$ and, in a finite
von Neumann algebra, $\ker T$ and $\ker T^*$ have the same trace, this
happens if and only if $\ker(S-z)\neq\{0\}$.

If $S-z$ is invertible, then $H_z$ is invertible, so $0$ lies in an open
component of $\mathbb R\setminus\supp\mu_{H_z}$ and is a nondegenerate
regular point of $H_z$; the infimum is attained by Corollary~\ref{cor:regularized-value-limit}.
Finally, $z\mapsto\log\Delta(S-z)$ is the logarithmic potential of the
Brown measure of $S$, which is recovered from it as
$\frac1{2\pi}\nabla_z^2\log\Delta(S-z)$ in the sense of distributions
\cite{Brown1986}.
\end{proof}

\section{The centered case and Gurvits capacity} 
\label{sec:capacity}

In the centered case $a=0$ and $z=0$, the Legendre term disappears:
for every $b\in\Fset$,
$
   \sup_{u=u^*}
   \left\{
      -\frac12\ip{u}{L_b^+ u}
   \right\}
   =
   0.
$
Hence
\begin{equation}\label{eq:centered-F}
   F_0(b)
   =
   \tr_m\log b
   =
   \log\Delta(b).
\end{equation}

Recall Gurvits' capacity \cite{Gurvits2004}, 
\begin{equation}\label{eq:capacity}
   \operatorname{Cap}(\eta)
   :=
   \inf
   \left\{
      \det\eta(X):
      X\succ0,\ \det X=1
   \right\},
\end{equation}
see also \cite{GargGurvitsOliveiraWigderson2020}.

We use the standard scaling rule
\begin{equation}\label{eq:capacity-scaling}
   \operatorname{Cap}
   \bigl(
      \mathcal S_{c_1,c_2}(\eta)
   \bigr)
   =
   |\det c_1|^2
   |\det c_2|^2
   \operatorname{Cap}(\eta),
\end{equation}
where
\[
   \mathcal S_{c_1,c_2}(\eta)(X)
   :=
   c_1\eta(c_2Xc_2^*)c_1^*.
\]

We first recall a standard fact.

\begin{lemma}[Capacity of a doubly stochastic map]
\label{lem:ds-capacity}
Let $\Phi:M_m(\mathbb C)\to M_m(\mathbb C)$ be completely positive,
unital, and trace-preserving.  Then
$
   \operatorname{Cap}(\Phi)=1.
$
\end{lemma}

\begin{proof}
Since $\Phi(I)=I$, the competitor $X=I$ gives
$
   \operatorname{Cap}(\Phi)\leq1.
$

Conversely, let $X\succ0$ with $\det X=1$.  Since $\log$ is operator
concave and $\Phi$ is unital and completely positive, the operator Jensen
inequality \cite{HansenPedersen1982} gives
\[
   \log\Phi(X)
   \succeq
   \Phi(\log X).
\]
Taking normalized traces and using trace preservation,
\[
   \tr_m\log\Phi(X)
   \geq
   \tr_m\Phi(\log X)
   =
   \tr_m\log X
   =
   0.
\]
Hence
$
   \det\Phi(X)\geq1.
$
Taking the infimum over all $X\succ0$ with $\det X=1$ gives
$
   \operatorname{Cap}(\Phi)\geq1.
$
\end{proof}

The centered finite-dimensional problem is exactly Gurvits' capacity.

\begin{theorem}[Centered variational problem and capacity]
\label{thm:centered-capacity}
For every self-adjoint completely positive covariance map $\eta$,
\begin{equation}\label{eq:centered-capacity}
   \boxed{
   \inf_{b\in\Fset}\Delta(b)
   =
   \operatorname{Cap}(\eta)^{1/(2m)}.
   }
\end{equation}

Moreover, the infimum in \eqref{eq:centered-capacity} is attained if and
only if $\eta$ is symmetrically doubly-stochastic scalable; equivalently,
if and only if there exists $C\succ0$ such that
\begin{equation}\label{eq:symmetric-DS}
   \eta(C)=C^{-1}.
\end{equation}
If such $C$ exists, an optimizer is
$
   b=C^{-1}.
$
\end{theorem}

\begin{proof}
We first prove the lower bound.

Let $b\in\Fset$.  Recall
$
   \widehat\eta_b
   =
   \mathcal C_{b^{-1/2}}
   \eta
   \mathcal C_{b^{-1/2}}
   =
   \mathcal S_{b^{-1/2},b^{-1/2}}(\eta).
$
By the scaling rule \eqref{eq:capacity-scaling},
\begin{equation}\label{eq:capacity-scaled-b}
   \operatorname{Cap}(\widehat\eta_b)
   =
   \det(b)^{-2}
   \operatorname{Cap}(\eta).
\end{equation}

Taking $X=I$ in the definition of capacity gives
$
   \operatorname{Cap}(\widehat\eta_b)
   \leq
   \det\widehat\eta_b(I).
$
Since $b\in\Fset$, Lemma~\ref{lem:feasible-pf} gives
$
   I-\widehat\eta_b\succeq0
$
as an operator on Hilbert--Schmidt space.  Pairing this inequality with
$I$ yields
\[
   \tr_m\widehat\eta_b(I)
   =
   \ip{I}{\widehat\eta_b(I)}
   \leq
   \ip{I}{I}
   =
   1.
\]
Moreover, $\widehat\eta_b(I)\succeq0$ as a matrix.  Hence the
arithmetic--geometric mean inequality gives
\[
   \det\widehat\eta_b(I)
   \leq
   \bigl(
      \tr_m\widehat\eta_b(I)
   \bigr)^m
   \leq1.
\]
Therefore
$
   \operatorname{Cap}(\widehat\eta_b)\leq1.
$
Using \eqref{eq:capacity-scaled-b},
$
   \operatorname{Cap}(\eta)
   \leq
   \det(b)^2.
$
Since
$
   \Delta(b)=\det(b)^{1/m},
$
we obtain
\[
   \Delta(b)
   \geq
   \operatorname{Cap}(\eta)^{1/(2m)}.
\]
Taking the infimum over $b\in\Fset$ gives
\begin{equation}\label{eq:capacity-lower}
   \inf_{b\in\Fset}\Delta(b)
   \geq
   \operatorname{Cap}(\eta)^{1/(2m)}.
\end{equation}

We next prove the reverse inequality.

Suppose first that $\eta$ is symmetrically doubly-stochastic scalable.
Choose $C\succ0$ satisfying
$
   \eta(C)=C^{-1}
$
and put
$
   b:=C^{-1}.
$
Then
$
   \widehat\eta_b
   =
   \mathcal S_{C^{1/2},C^{1/2}}(\eta)
$
is unital:
\[
   \widehat\eta_b(I)
   =
   C^{1/2}\eta(C)C^{1/2}
   =
   I.
\]
The map $\widehat\eta_b$ is Hilbert--Schmidt self-adjoint, so unitality
also implies trace preservation:
\[
\begin{aligned}
   \tr_m\widehat\eta_b(X)
   &=
   \ip{I}{\widehat\eta_b(X)}\\
   &=
   \ip{\widehat\eta_b(I)}{X}\\
   &=
   \tr_m X.
\end{aligned}
\]
Thus $\widehat\eta_b$ is doubly stochastic.

By the Kadison--Schwarz inequality,
\[
   \widehat\eta_b(X)^*
   \widehat\eta_b(X)
   \preceq
   \widehat\eta_b(X^*X).
\]
Taking normalized traces and using trace preservation,
$
   \|\widehat\eta_b(X)\|_2^2
   \leq
   \|X\|_2^2.
$
Hence
$
   \|\widehat\eta_b\|_{\mathrm{HS}\to\mathrm{HS}}\leq1,
$
and therefore $b\in\Fset$.

By Lemma~\ref{lem:ds-capacity},
$
   \operatorname{Cap}(\widehat\eta_b)=1.
$
Using \eqref{eq:capacity-scaled-b},
$
   1
   =
   \det(b)^{-2}\operatorname{Cap}(\eta),
$
so
$
   \det(b)^2
   =
   \operatorname{Cap}(\eta).
$
Consequently,
\[
   \Delta(b)
   =
   \operatorname{Cap}(\eta)^{1/(2m)}.
\]
Thus equality holds in \eqref{eq:capacity-lower} whenever an exact
symmetric doubly-stochastic scaling exists.

For a general $\eta$, introduce the strictly positive regularization
\begin{equation}\label{eq:capacity-regularization}
   \eta_\varepsilon(X)
   :=
   \eta(X)
   +
   \varepsilon\,\tr_m(X)I,
   \qquad
   \varepsilon>0.
\end{equation}
Each $\eta_\varepsilon$ is self-adjoint, completely positive, and strictly
positive; hence it is symmetrically doubly-stochastic scalable, \cite{KarginNonSingular2026}.  Denote its
feasible set by
\[
   \Fset_\varepsilon
   :=
   \left\{
      b\succ0:
      \mathcal C_b-\eta_\varepsilon\succeq0
   \right\}.
\]
Since
$
   \eta_\varepsilon\succeq\eta,
$
we have
$
   \Fset_\varepsilon\subseteq\Fset.
$
By the already proved scalable case,
\[
   \inf_{b\in\Fset_\varepsilon}\Delta(b)
   =
   \operatorname{Cap}(\eta_\varepsilon)^{1/(2m)}.
\]
Therefore
\begin{equation}\label{eq:capacity-upper-eps}
   \inf_{b\in\Fset}\Delta(b)
   \leq
   \operatorname{Cap}(\eta_\varepsilon)^{1/(2m)}.
\end{equation}

It remains to let $\varepsilon\downarrow0$.  We claim that
\begin{equation}\label{eq:capacity-continuity}
   \operatorname{Cap}(\eta_\varepsilon)
   \longrightarrow
   \operatorname{Cap}(\eta).
\end{equation}
Since
$
   \eta_\varepsilon(X)\succeq\eta(X),$ with$X\succ0$,
monotonicity of the determinant on the positive cone gives
$
   \operatorname{Cap}(\eta_\varepsilon)
   \geq
   \operatorname{Cap}(\eta).
$

Conversely, fix $\delta>0$ and choose $X_\delta\succ0$ with
$
   \det X_\delta=1
$
and
\[
   \det\eta(X_\delta)
   \leq
   \operatorname{Cap}(\eta)+\delta.
\]
For this fixed $X_\delta$,
$
   \eta_\varepsilon(X_\delta)
   \longrightarrow
   \eta(X_\delta),
$
so
$
   \det\eta_\varepsilon(X_\delta)
   \longrightarrow
   \det\eta(X_\delta).
$
Since
\[
   \operatorname{Cap}(\eta_\varepsilon)
   \leq
   \det\eta_\varepsilon(X_\delta),
\]
we obtain
\[
   \limsup_{\varepsilon\downarrow0}
   \operatorname{Cap}(\eta_\varepsilon)
   \leq
   \operatorname{Cap}(\eta)+\delta.
\]
Letting $\delta\downarrow0$ proves
\eqref{eq:capacity-continuity}.

Passing to the limit in
\eqref{eq:capacity-upper-eps} gives
\[
   \inf_{b\in\Fset}\Delta(b)
   \leq
   \operatorname{Cap}(\eta)^{1/(2m)}.
\]
Together with \eqref{eq:capacity-lower}, this proves
\eqref{eq:centered-capacity}.

It remains to characterize attainment.  Suppose that
$b\in\Fset$ attains the infimum.  Then
$
   \det(b)^2
   =
   \operatorname{Cap}(\eta),
$
and therefore, by \eqref{eq:capacity-scaled-b},
$
   \operatorname{Cap}(\widehat\eta_b)=1.
$
In the chain
\[
   \operatorname{Cap}(\widehat\eta_b)
   \leq
   \det\widehat\eta_b(I)
   \leq
   \bigl(
      \tr_m\widehat\eta_b(I)
   \bigr)^m
   \leq
   1,
\]
all inequalities must be equalities.  In particular,
$
   \tr_m\widehat\eta_b(I)=1
$
and
$
   \det\widehat\eta_b(I)=1.
$
Equality in the arithmetic--geometric mean inequality forces
$
   \widehat\eta_b(I)=I.
$
Since $\widehat\eta_b$ is self-adjoint, it is therefore doubly stochastic.

Writing
$
   C:=b^{-1},
$
the identity
$
   \widehat\eta_b(I)=I
$
is exactly
$
   C^{1/2}\eta(C)C^{1/2}=I,
$
or
$
   \eta(C)=C^{-1}.
$
Thus $\eta$ is symmetrically doubly-stochastic scalable.

The converse was proved above: every symmetric doubly-stochastic scaler
produces an optimizer.
\end{proof}

\begin{corollary}[Mai--Speicher formula]
\label{cor:mai-speicher}
For the centered matrix semicircle
\[
   S_0
   =
   \sum_{j=1}^r A_j\otimes s_j,
\]
one has
\begin{equation}\label{eq:mai-speicher}
   \boxed{
   \log\Delta(S_0)
   =
   -\frac12
   +
   \frac1{2m}
   \log\operatorname{Cap}(\eta).
   }
\end{equation}
The identity is understood in the extended-real sense when
$\operatorname{Cap}(\eta)=0$.
\end{corollary}

\begin{proof}
By \eqref{eq:centered-F},
\[
   \inf_{b\in\Fset}F_0(b)
   =
   \inf_{b\in\Fset}\log\Delta(b).
\]
Theorem~\ref{thm:centered-capacity} gives
\[
   \inf_{b\in\Fset}\log\Delta(b)
   =
   \frac1{2m}
   \log\operatorname{Cap}(\eta).
\]
Substituting this into Theorem~\ref{thm:real-axis}, applied with
$a=0$ at the point $x=0$, gives
\eqref{eq:mai-speicher}.
\end{proof}

\begin{remark}[Attainment versus escape]
\label{rem:capacity-attainment}
The theorem separates two qualitatively different centered regimes.

If $\eta$ is symmetrically doubly-stochastic scalable, the capacity
minimum is realized by a finite positive scale.

If $\operatorname{Cap}(\eta)>0$ but $\eta$ is not symmetrically
doubly-stochastic scalable, then
\[
   \inf_{b\in\Fset}\Delta(b)>0
\]
but the infimum is not attained.  Any minimizing sequence must therefore
leave every compact subset of the positive definite cone; otherwise a
convergent subsequence would produce an optimizer.

Thus non-attainment of the centered variational problem is precisely the
failure of exact symmetric doubly-stochastic scaling.
\end{remark}

\section{Further questions}
\label{sec:further-questions}

We collect some questions suggested by the variational formula.

\subsection*{Attainment}
Corollary~\ref{cor:regularized-value-limit} gives attainment at nondegenerate regular points,
and Remark~\ref{rem:attainment} shows that attainment at any point where
$G(x+i0)$ exists would follow from boundedness of the regularized scales
$b_\varepsilon$ along a sequence.  We do not know whether $b_\varepsilon$ stays bounded at
degenerate regular points, where $-\operatorname{Im}G(x+i0)$ is nonzero
but singular, or at spectral edges.  In Example~\ref{exa:coupled} it does,
numerically, with corrections of order $\varepsilon$.  At non-atomic
singular points, where $\|G(x+i\varepsilon)\|\to\infty$, the regularized
scales escape, and we do not know whether the finite infimum is attained
by some other positive scale.  In the centered case the answer is given by
Theorem~\ref{thm:centered-capacity}: attainment holds if and only if
$\eta$ is symmetrically doubly-stochastic scalable.

\subsection*{The active set and the spectrum}
In the scalar model of Section~\ref{sec:scalar-model} the minimizer lies
on $\partial\Fset$ inside the support and in the interior of $\Fset$
outside it.  Proposition~\ref{prop:active-set} shows that the same
dichotomy holds in the matrix case at nondegenerate regular points:
$b_z\in\partial\Fset$ when $-\operatorname{Im}G(z+i0)\succ0$, and
$b_z\in\operatorname{int}\Fset$ when $z$ is in the resolvent set.  What
remains open is the behaviour at the points of the support where
$G(z+i0)$ exists and is self-adjoint, such as spectral edges and
interior zeros of the density.  

\begin{question}
\label{ques:active-set}
Let $z\in\supp\mu_{S_a}$ be a point at which $G(z+i0)$ exists and is
self-adjoint.  Do the regularized scales $b_\varepsilon$ converge, and
does the limit lie on $\partial\Fset$?  In particular, does
$z\in\supp\mu_{S_a}$ coincide with $b_z\in\partial\Fset$ at every point
where $b_z$ exists?
\end{question}
For the hermitized problem of Corollary~\ref{cor:brown} the same question
has an additional twist.  If $S-z$ is invertible, then $0$ lies in the
resolvent set of $H_z$, and Proposition~\ref{prop:active-set}(ii) gives
$b_z^{\mathrm h}\in\operatorname{int}\Fset^{\mathrm h}$; conversely, if
$0$ is a nondegenerate regular point of $H_z$ in the support of
$\mu_{H_z}$, then $b_z^{\mathrm h}\in\partial\Fset^{\mathrm h}$.  Thus the
active set of the hermitized problem tracks the spectrum of $S$.  

The support of the Brown measure is always contained in the spectrum, and
for general operators the inclusion can be strict: for an $R$-diagonal
element $T$ this happens exactly when $T^{-1}$ belongs to $L^2$ but is
unbounded \cite[Proposition~4.6]{HaagerupLarsen2000}.  This mechanism is
not available for matrix semicirculars, since
for every $z\in\mathbb C$, $S-z$ is invertible if and only if it is
injective and $(\tr_m\otimes\tau)\bigl(|S-z|^{-2}\bigr)<\infty$.

Indeed, let $\mathcal A$ be the $*$-algebra of operators affiliated with
the von Neumann algebra generated by $s_1,\dots,s_r$.  An element of
$M_m(\mathcal A)$ is invertible if and only if it is injective
\cite[Section~5.1]{MaiSpeicherYin2023}.  If $S-z$ is injective, the
entries of $(S-z)^{-1}$ belong, by definition, to the rational closure of
$\mathbb C\langle s_1,\dots,s_r\rangle$ in $\mathcal A$; for free
semicirculars this coincides with the division closure $D(s)$ and is a
skew field \cite[Theorem~1.1 and Remark~5.10]{MaiSpeicherYin2023}.  By
\cite[Theorem~1.1]{Miyagawa2026}, $D(s)\cap L^2\subset L^\infty$.  Hence,
if the entries of $(S-z)^{-1}$ are square integrable, they are bounded.

Whether the spectrum and the support of the Brown measure coincide for
rational functions in free semicirculars is asked in \cite{Miyagawa2026};
for the sum of a self-adjoint element and a free elliptic element they do
coincide \cite{HallHo2025}.

\begin{question}
\label{ques:brown-active}
Let $S=a_0\otimes1+\sum_ia_i\otimes s_i$ with arbitrary
$a_i\in M_m(\mathbb C)$.
Does the spectrum of $S$ coincide with the support of its Brown
measure?  In variational terms: is the Brown measure supported exactly on
the set of $z$ at which the constraint
$\mathcal C_b-\eta^{\mathrm h}\succeq0$ is active at the minimizer $b^{\mathrm h}_z = \lim_{\eps \downarrow 0} b_\eps$?

\end{question}

\subsection*{Singular points and the escape rate}
At a non-atomic singular point the boundary behaviour of $G$ is governed
by a Puiseux exponent $\theta\in(0,1)$ (Appendix~\ref{app:local-structure}):
$\|G(x+i\varepsilon)\|\asymp\varepsilon^{-\theta}$, and the density of
$\mu_{S_a}$ is $O(|t-x|^{-\theta})$ near $x$.  The regularized scales
satisfy $\|b_\varepsilon^{-1}\|\geq\|G(x+i\varepsilon)\|$.
It would be interesting to know whether the exponent $\theta$, and more
generally the Puiseux data of $G$ at $x$, can be read off from the
geometry of the escaping minimizing sequence, for instance from the rates
at which the eigenvalues of $b_\varepsilon$ tend to $0$ and to $\infty$.
Their product is constrained by $\det b_\varepsilon=|\det G(x+i\varepsilon)|^{-1}$.

\subsection*{Other free families}
The proof uses the semicircular structure in one place: the covariance
map $\eta$ is linear, so the Dyson equation is quadratic and the Dyson
action is explicit.  Lehner's formulas \cite{Lehner1999} for the spectral
edge exist both for free semicirculars and for free Haar unitaries, and
Mai and Speicher \cite[Question~1]{MaiSpeicher2024} ask for which
families of operators the determinantal property
$\Delta(X)\geq1$ holds for invertible $X=\sum_ia_i\otimes x_i$ with
integer coefficients.

\begin{question}
\label{ques:other-families}
Is there a variational formula analogous to \eqref{eq:intro-main-formula}
for $a\otimes1+\sum_ia_i\otimes u_i$ with free Haar unitaries $u_i$, or
more generally for operator-valued free variables whose $R$-transform is a
polynomial in the operator-valued sense?
\end{question}

\subsection*{Random matrices}
Mai and Speicher \cite[Conjecture~1]{MaiSpeicher2024} conjecture that for
independent GUE matrices $X_1^{(N)},\dots,X_r^{(N)}$ the normalized
log-determinants of the block matrices $A_N=\sum_iA_i\otimes X_i^{(N)}$
converge to $\log\Delta(S_0)$.  The shifted analogue would be
\[
   \frac1{mN}\log\bigl|\det(A_N-z)\bigr|
   \longrightarrow
   -\frac12+\inf_{b\in\Fset}F_z(b)
   \qquad(z\in\mathbb R),
\]
and, via Corollary~\ref{cor:brown}, its non-Hermitian counterpart for the
Brown measure would identify the limiting empirical spectral distribution
of $\sum_ia_i\otimes X_i^{(N)}$ with arbitrary $a_i$.  The difficulty is
the usual one for log-determinants, the contribution of the smallest
singular values, and the variational formula does not by itself address
it.  It does, however, provide an explicit and computable target.

\subsection*{Computation and duality}
Kunisky \cite{Kunisky2026} reformulated Lehner's formulas as semidefinite
programs and used complementary slackness to obtain sensitivity
information about the optimizer.  The variational problem
\eqref{eq:intro-main-formula} is geodesically convex rather than convex,
and the feasible set is defined by the operator inequality
$\mathcal C_b\succeq\eta$, which is quadratic in $b$.

\begin{question}
\label{ques:sdp}
Does \eqref{eq:intro-main-formula} admit a convex reformulation, in the
manner of \cite{Kunisky2026}, and does its dual carry information about
the spectral density, for instance the derivative of $\log\Delta(S_a-z)$
in the coefficients $a$ and $A_j$?
\end{question}

The first variation of Section~\ref{sec:geometry} suggests that the dual
object is the Perron eigenvector $x_0=b_z^{1/2}\gamma\,b_z^{1/2}$ of
$\widehat\eta_{b_z}$, which already appears in the proof of
Proposition~\ref{prop:active-set}.

\subsection*{An unconstrained formula}
Sheng Yin has posted a different variational expression for the same
determinant \cite[Theorem~1.3]{Yin2026}.  For
$S=a_0\otimes1+\sum_ja_j\otimes s_j$ with arbitrary
$a_0,a_j\in M_m(\mathbb C)$, put $\eta(b)=\sum_ja_jba_j^*$ and
$\rho(b)=\sum_ja_jba_j$, for $b\succ0$ let $Z_b$ be the solution of
the real-linear equation $(I+\eta(b))Zb^{-1}+\rho(Z^*)=a_0$, and set
\[
   \Phi(b)
   :=
   \frac1{2m}\Bigl[\log\det(I+\eta(b))-\log\det b
      +\operatorname{Re}\operatorname{Tr}(a_0^*Z_b)\Bigr].
\]
The statement is that
\begin{equation}\label{eq:yin-formula}
   \log\Delta(S)=-\frac12+\inf_{b\succ0}\Phi(b),
\end{equation}
and that the infimum is attained if and only if $S$ is boundedly
invertible, in which case the minimizer is unique.  Here the infimum is
over all positive definite matrices of size $m$, with no constraint, and
no hermitization is needed.  According to \cite{Yin2026}, the text of that
manuscript was generated by a large language model and has not been
checked by its author.  We have not verified its proofs in full; the
formula \eqref{eq:yin-formula} agrees with ours in numerical tests.

For a self-adjoint matrix semicircle and real $z$, that is, for
$a_0=a_z$ and $a_j=A_j$, the two variational problems are related as
follows \cite[Section~4]{Yin2026}.  The map $b\mapsto c_b$, where
$c_b\succ0$ is the solution of $c_b\,b\,c_b=I+\eta(b)$, is a bijection
from the positive definite cone onto $\Fint$, with inverse
$c\mapsto(L_c^-)^{-1}(I)$.  Under this bijection the determinant terms
of $\Phi$ and $F_z$ coincide, and the quadratic terms are ordered, so
that
\begin{equation}\label{eq:yin-comparison}
   \Phi(b)\ \ge\ F_z(c_b)\qquad(b\succ0),
\end{equation}
with strict inequality in general.  By Theorem~\ref{thm:real-axis},
\eqref{eq:yin-comparison} gives
$\inf_{b\succ0}\Phi(b)\ge\log\Delta(S_a-z)+\frac12$.  Conversely, by
Remark~\ref{rem:yin-dictionary}, the matrices
$\gamma_\eps/\eps$ satisfy
$\Phi(\gamma_\eps/\eps)\le F_{z,\eps}(b_\eps)
=\log\Delta(S_a-z-i\eps)+\frac12$, and the right-hand side decreases to
$\log\Delta(S_a-z)+\frac12$.  Thus the value in \eqref{eq:yin-formula}
follows from Theorem~\ref{thm:real-axis} in the self-adjoint case.  In
the opposite direction, \cite[Proposition~1.4]{Yin2026} deduces
\eqref{eq:intro-main-formula}, with the infimum over $\Fint$, from
\eqref{eq:yin-formula}, \eqref{eq:yin-comparison} and a separate upper
bound.

The two problems behave differently.  The function $F_z$ is
geodesically convex, and inside the spectrum its minimizer, when it
exists, lies on $\partial\Fset$.  The function $\Phi$ has no minimizer
inside the spectrum, and it is not geodesically convex: for $m=1$ and
$S=a_0+\sigma s$ one finds
$2\Phi(b)=\log(\sigma^2+b^{-1})+a_0^2\,b/(1+2\sigma^2b)$, which is
not convex in $\log b$ when $a_0^2>27\sigma^2/4$.

\begin{question}
\label{ques:unconstrained}
For non-self-adjoint coefficients, \eqref{eq:yin-formula} is a problem
in size $m$, while Corollary~\ref{cor:brown} is a geodesically convex
problem in size $2m$.  Is there an analogue of the bijection
$b\mapsto c_b$ and of the comparison \eqref{eq:yin-comparison} relating
them?  Is there a constrained, geodesically convex formula for
$\log\Delta(S-z)$ in size $m$?
\end{question}

\appendix
\section{Limits of the regularized normalizers}
\label{app:normalizer-limits}

This appendix proves the convergence statement for the regularized
unitary normalizers used in Section~\ref{sec:attainment}.  The argument
is purely finite-dimensional and depends only on the explicit $pUp$
factorization.

Fix $z\in\mathbb R$ and write
\[
   G_\varepsilon
   :=
   G(z+i\varepsilon)
   =
   X_\varepsilon-iY_\varepsilon,
   \qquad
   Y_\varepsilon\succ0.
\]
Let $b_\varepsilon$ be the unitary normalizer given by
Lemma~\ref{lem:pup}.  Thus
\begin{equation}\label{eq:app-pup-normalizer}
   b_\varepsilon^{-1}
   =
   Y_\varepsilon^{1/2}
   \left(
      I+
      \bigl(
         Y_\varepsilon^{-1/2}
         X_\varepsilon
         Y_\varepsilon^{-1/2}
      \bigr)^2
   \right)^{1/2}
   Y_\varepsilon^{1/2}.
\end{equation}

We now prove Proposition~\ref{prop:normalizer-limit}.

\begin{proof}
We treat the two cases separately.

\smallskip

\noindent
\emph{Case 1: strictly negative imaginary part.}
Assume
$
   Y:=-\operatorname{Im}G\succ0.
$
Then
\[
   X_\varepsilon\longrightarrow
   X:=\operatorname{Re}G,
   \qquad
   Y_\varepsilon\longrightarrow Y.
\]
Since $Y\succ0$, the maps
$
   Y\longmapsto Y^{\pm1/2}
$
are continuous in a neighborhood of $Y$.  Hence every factor in
\eqref{eq:app-pup-normalizer} converges, and therefore
\[
   b_\varepsilon^{-1}
   \longrightarrow
   Y^{1/2}
   \left(
      I+
      \bigl(
         Y^{-1/2}XY^{-1/2}
      \bigr)^2
   \right)^{1/2}
   Y^{1/2}.
\]
The right-hand side is precisely the inverse of the unitary normalizer
obtained by applying Lemma~\ref{lem:pup} directly to $G$.  This proves
part~(i).

\smallskip

\noindent
\emph{Case 2: a real resolvent gap.}
Suppose now that $z$ belongs to an open component of
$\mathbb R\setminus\supp\mu_{S_a}$.  Then the operator-valued Cauchy
transform extends analytically across a neighborhood of $z$, and
$
   G(z)=G(z)^*.
$
Moreover,
\[
   G'(z)
   =
   -
   \int_{\mathbb R}
   \frac{1}{(z-t)^2}\,dv(t),
\]
where $v$ is the $M_m(\mathbb C)$-valued spectral measure of $S_a$.
Thus
\begin{equation}\label{eq:app-D}
   D:=-G'(z)
   =
   \int_{\mathbb R}
   \frac{1}{(z-t)^2}\,dv(t)
   \succ0.
\end{equation}
Indeed, for every nonzero vector $\xi$,
\[
   \langle\xi,D\xi\rangle
   =
   \int_{\mathbb R}
   \frac{1}{(z-t)^2}\,
   d\langle\xi,v(t)\xi\rangle
   >0.
\]

By Schwarz reflection,
$
   G(\bar w)=G(w)^*,
$
so the Taylor expansion of $G(z+i\varepsilon)$ at $z$ has the form
\begin{align}
   X_\varepsilon
   &=
   G(z)+O(\varepsilon^2),
   \label{eq:app-X-expansion}\\
   Y_\varepsilon
   &=
   \varepsilon D+O(\varepsilon^3).
   \label{eq:app-Y-expansion}
\end{align}
It is convenient to write
$
   D_\varepsilon
   :=
   \varepsilon^{-1}Y_\varepsilon.
$
Then
$
   D_\varepsilon\longrightarrow D\succ0.
$

Define
\[
   T_\varepsilon
   :=
   D_\varepsilon^{-1/2}
   X_\varepsilon
   D_\varepsilon^{-1/2}.
\]
Since
$
   Y_\varepsilon
   =
   \varepsilon D_\varepsilon,
$
we have
\[
   Y_\varepsilon^{-1/2}
   X_\varepsilon
   Y_\varepsilon^{-1/2}
   =
   \varepsilon^{-1}T_\varepsilon.
\]
Substituting this into \eqref{eq:app-pup-normalizer},
\begin{align*}
   b_\varepsilon^{-1}
   &=
   \varepsilon^{1/2}D_\varepsilon^{1/2}
   \left(
      I+\varepsilon^{-2}T_\varepsilon^2
   \right)^{1/2}
   \varepsilon^{1/2}D_\varepsilon^{1/2}\\
   &=
   D_\varepsilon^{1/2}
   \left(
      \varepsilon^2I+T_\varepsilon^2
   \right)^{1/2}
   D_\varepsilon^{1/2}.
\end{align*}
Now
\[
   T_\varepsilon
   \longrightarrow
   T
   :=
   D^{-1/2}G(z)D^{-1/2}.
\]
Since $G(z)$ is invertible, so is $T$.  Therefore, by continuity of the
functional calculus,
$
   \left(
      \varepsilon^2I+T_\varepsilon^2
   \right)^{1/2}
   \longrightarrow
   |T|.
$
Consequently,
\[
   b_\varepsilon^{-1}
   \longrightarrow
   D^{1/2}|T|D^{1/2},
\]
which is exactly \eqref{eq:gap-normalizer}.

It remains to verify that the limiting matrix is indeed a unitary
normalizer of $G(z)$.  Since $T=T^*$ is invertible,
\[
   T
   =
   |T|^{1/2}
   J
   |T|^{1/2},
   \qquad
   J:=\operatorname{sign}(T),
\]
where $J=J^*$ and $J^2=I$.  Put
$
   L
   :=
   D^{1/2}|T|^{1/2}.
$
Then
\[
   G(z)
   =
   D^{1/2}TD^{1/2}
   =
   LJL^*,
\]
and
\[
   LL^*
   =
   D^{1/2}|T|D^{1/2}
   =
   b_z^{-1}.
\]
Take the polar decomposition
  $
   L=b_z^{-1/2}V,$
   where $
   V\in U(m).$
It follows that
\[
   G(z)
   =
   b_z^{-1/2}
   (VJV^*)
   b_z^{-1/2}.
\]
Since $VJV^*$ is unitary,
\[
   b_z^{1/2}G(z)b_z^{1/2}
   =
   VJV^*
\]
is unitary.  Thus $b_z$ is a unitary normalizer of $G(z)$.

Finally, if $G(z)$ is definite, then $T$ is definite as well.  Hence
$
   |T|
   =
   \pm T,
$
with the sign chosen according to the sign of $G(z)$.  Therefore
\[
   b_z^{-1}
   =
   D^{1/2}|T|D^{1/2}
   =
   \pm G(z)
   =
   |G(z)|,
\]
which proves \eqref{eq:definite-normalizer}.
\end{proof}

\begin{remark}
The appearance of
\[
   D=-G'(z)
\]
in \eqref{eq:gap-normalizer} reflects the fact that when the boundary
value $G(z)$ is self-adjoint and indefinite, it may admit more than one
positive unitary normalizer.  The vertical regularization
$G(z+i\varepsilon)$ selects a canonical one, and this selection depends on
the first transverse derivative of the Cauchy transform at the real axis.

When $G(z)$ is definite, the ambiguity disappears and the selected
normalizer reduces to the familiar expression
$
   |G(z)|^{-1}.
$
\end{remark}

\section{Local regularity at non-atomic points}
\label{app:local-structure}

The purpose of this appendix is to establish the local
regularity needed to characterize the value $-\infty$
in the main theorem. We prove that every non-atomic real
point has a neighborhood on which the spectral measure
has a density in $L^p$ for some $p>1$.
By H\"older's inequality, this implies local logarithmic
integrability and hence finiteness of the logarithmic
potential at that point; see
Corollary~\ref{cor:minus-infinity-atoms}.
The real-axis variational identity itself does not
require this regularity input.
Let
\[
   G(w)
   =
   (\id\otimes\tau)
   \bigl[(wI-S_a)^{-1}\bigr],
   \qquad
   w\in\mathbb C^+,
\]
and let
\[
   g(w):=\tr_m G(w)
   =
   \int_{\mathbb R}\frac{1}{w-t}\,d\mu_{S_a}(t).
\]

We use the entrywise algebraicity of the matrix Cauchy transform
proved in \cite[Proposition~2.4]{KarginSingularities2026};
compare also \cite{ShlyakhtenkoSkoufranis2015} for algebraicity of
scalar Cauchy transforms of matricial polynomials in free semicirculars.

\subsection{Local Puiseux control}

The only consequence of algebraicity that we need is the following local
estimate.

\begin{lemma}[Local Puiseux control]
\label{lem:local-puiseux-control}
Fix $x\in\mathbb R$.  There exist an integer $N\geq1$, a radius $r>0$,
matrices $C_k\in M_m(\mathbb C)$, and an integer $k_0$ with
$C_{k_0}\neq0$ such that the branch of $G$ selected from the upper
half-plane has a convergent Puiseux expansion
\begin{equation}\label{eq:local-puiseux}
   G(x+\zeta)
   =
   \sum_{k\geq k_0}
   C_k\zeta^{k/N},
   \qquad
   0<|\zeta|<r,
   \quad
   0<\arg\zeta<\pi.
\end{equation}
Here the fractional powers are taken with the branch determined by
$0<\arg\zeta<\pi$.

Set
$
   \theta(x):=-\frac{k_0}{N}.
$
Then
$
   0\leq\theta(x)\leq 1
 $
and, after decreasing $r$ if necessary,
\begin{equation}\label{eq:local-G-bound}
   \|G(x+\zeta)\|
   \leq
   C|\zeta|^{-\theta(x)}
\end{equation}
throughout the upper half-disc
$
   0<|\zeta|<r,$
   $
   \operatorname{Im}\zeta>0.
$
\end{lemma}

\begin{proof}
By entrywise algebraicity, the classical Newton--Puiseux theorem \cite{Walker1950} gives a
convergent Puiseux expansion for every entry $G_{pq}$ on the local branch
selected from $\mathbb C^+$.  Taking a common multiple of the finitely
many ramification indices gives \eqref{eq:local-puiseux}.

Factoring out the lowest nonzero power gives
\[
   G(x+\zeta)
   =
   \zeta^{-\theta(x)}
   \left(
      C_{k_0}
      +
      \sum_{k>k_0}
      C_k\zeta^{(k-k_0)/N}
   \right).
\]
The series in parentheses is uniformly bounded for sufficiently small
$\zeta$ in the upper half-disc.  This gives
\eqref{eq:local-G-bound}.

It remains to prove  that  $0\leq\theta(x)\leq 1$.  The resolvent estimate gives
$
   \|G(x+i\varepsilon)\|
   \leq
   \frac1{\varepsilon}.
$
Comparison with the leading Puiseux behavior shows that
$
   \theta(x)\leq1.
$

Suppose, on the other hand, that $\theta(x)<0$.  Then
$
   G(x+i\varepsilon)\longrightarrow0.
$
The Dyson equation gives
\[
   G(w)^{-1}
   =
   wI-a-\eta(G(w)).
\]
Its right-hand side remains bounded as $w\to x$, whereas
$\|G(w)\|\to0$ would imply
$
   \|G(w)^{-1}\|\longrightarrow\infty.
$
This contradiction proves $\theta(x)\geq0$.
\end{proof}

\subsection{The exponent one and atoms}

The extremal exponent $\theta=1$ has a simple spectral meaning.

\begin{lemma}[Atoms and the exponent one]
\label{lem:theta-one-atom}
For $x\in\mathbb R$,
\begin{equation}\label{eq:theta-one-iff-atom}
   \theta(x)=1
   \quad\Longleftrightarrow\quad
   \mu_{S_a}(\{x\})>0.
\end{equation}
\end{lemma}

\begin{proof}
Suppose first that $\theta(x)=1$.  The Puiseux expansion gives
$
   i\varepsilon G(x+i\varepsilon)
   \longrightarrow
   P
$
for some nonzero matrix $P$.

On the other hand, by the spectral theorem,
\[
   i\varepsilon
   (x+i\varepsilon-S_a)^{-1}
   \longrightarrow
   \mathbf 1_{\{x\}}(S_a)
\]
strongly as $\varepsilon\downarrow0$.  Since the operators on the left
are uniformly bounded, applying $\id\otimes\tau$ gives
\[
   i\varepsilon G(x+i\varepsilon)
   \longrightarrow
   P_x
   :=
   (\id\otimes\tau)
   \bigl[\mathbf 1_{\{x\}}(S_a)\bigr].
\]
Hence
\[
   P=P_x\neq0.
\]
The matrix $P_x$ is positive semidefinite, so
$
   \mu_{S_a}(\{x\})
   =
   \tr_m P_x
   >0.
$

Conversely, suppose
$
   p:=\mu_{S_a}(\{x\})>0.
$
Since
\[
   -\operatorname{Im}g(x+i\varepsilon)
   =
   \int_{\mathbb R}
   \frac{\varepsilon}
        {(t-x)^2+\varepsilon^2}
   \,d\mu_{S_a}(t),
\]
the atom at $x$ gives
\[
   -\operatorname{Im}g(x+i\varepsilon)
   \geq
   \frac{p}{\varepsilon}.
\]
Because the normalized trace is dominated by the operator norm,
\[
   \|G(x+i\varepsilon)\|
   \geq
   \frac{p}{\varepsilon}.
\]
Together with the resolvent bound
$
   \|G(x+i\varepsilon)\|
   \leq
   \frac1{\varepsilon},
$
this forces the leading Puiseux exponent to be
$
   \theta(x)=1.
$
\end{proof}

\subsection{Local $L^p$ regularity}

\begin{theorem}[Local regularity at non-atomic points]
\label{thm:nonatomic-local-regularity}
Let $x\in\mathbb R$ satisfy
$
   \mu_{S_a}(\{x\})=0.
$
Then there exist an open interval $I\ni x$, a number $p>1$, and a
nonnegative function
$
   \rho\in L^p(I)
$
such that
\begin{equation}\label{eq:local-ac}
   d\mu_{S_a}(t)\big|_I
   =
   \rho(t)\,dt.
\end{equation}

More precisely, if $\theta=\theta(x)$ is the exponent from
Lemma~\ref{lem:local-puiseux-control}, then $\theta<1$ and one may choose
$I$ so that
\begin{equation}\label{eq:rho-pointwise-bound}
   0\leq
   \rho(t)
   \leq
   C|t-x|^{-\theta}
   \qquad
   \text{for a.e. }t\in I.
\end{equation}
Consequently,
\[
   \rho\in L^p(I)
   \qquad
   \text{for every }
   1<p<\frac1\theta
\]
when $\theta>0$, while for $\theta=0$ one has
$
   \rho\in L^\infty(I).
$
\end{theorem}

\begin{proof}
By Lemma~\ref{lem:theta-one-atom}, non-atomicity implies
$
   \theta:=\theta(x)<1.
$
Choose $I\ni x$ sufficiently small that the estimate
\[
   \|G(t+i\varepsilon)\|
   \leq
   C
   \bigl(
      (t-x)^2+\varepsilon^2
   \bigr)^{-\theta/2}
\]
holds for $t\in I$ and all sufficiently small $\varepsilon>0$.

Define the Poisson regularization of the spectral measure by
\[
   \rho_\varepsilon(t)
   :=
   -\frac1\pi
   \operatorname{Im}g(t+i\varepsilon).
\]
Since $g$ is the scalar Cauchy transform of $\mu_{S_a}$,
\begin{equation}\label{eq:poisson-density}
   \rho_\varepsilon(t)
   =
   \int_{\mathbb R}
   P_\varepsilon(t-s)\,d\mu_{S_a}(s),
\end{equation}
where
\[
   P_\varepsilon(s)
   :=
   \frac1\pi
   \frac{\varepsilon}{s^2+\varepsilon^2}.
\]
Moreover,
\[
   0\leq
   \rho_\varepsilon(t)
   \leq
   \frac1\pi\|G(t+i\varepsilon)\|.
\]
Hence
\begin{equation}\label{eq:poisson-domination}
   0\leq
   \rho_\varepsilon(t)
   \leq
   C'
   \bigl(
      (t-x)^2+\varepsilon^2
   \bigr)^{-\theta/2}
   \leq
   C'|t-x|^{-\theta}
\end{equation}
for $t\neq x$.

Let $\varphi\in C_c(I)$ be nonnegative.  From
\eqref{eq:poisson-density} and Fubini's theorem,
\[
\begin{aligned}
   \int_I
   \varphi(t)\rho_\varepsilon(t)\,dt
   &=
   \int_{\mathbb R}
   (P_\varepsilon*\varphi)(s)\,
   d\mu_{S_a}(s).
\end{aligned}
\]
Since
$
   P_\varepsilon*\varphi
   \longrightarrow
   \varphi
$
uniformly as $\varepsilon\downarrow0$, we obtain
\[
   \int_I\varphi(t)\,d\mu_{S_a}(t)
   =
   \lim_{\varepsilon\downarrow0}
   \int_I
   \varphi(t)\rho_\varepsilon(t)\,dt.
\]
Using \eqref{eq:poisson-domination},
\[
   \int_I\varphi(t)\,d\mu_{S_a}(t)
   \leq
   C'
   \int_I
   \varphi(t)|t-x|^{-\theta}\,dt.
\]
Therefore, as measures on $I$,
\begin{equation}\label{eq:measure-domination}
   \mu_{S_a}|_I
   \leq
   C'|t-x|^{-\theta}\,dt.
\end{equation}
It follows that
\[
   \mu_{S_a}|_I=\rho(t)\,dt
\]
for some measurable $\rho\geq0$ satisfying
\eqref{eq:rho-pointwise-bound}.

If $\theta=0$, this immediately gives
$
   \rho\in L^\infty(I).
$
If $0<\theta<1$, then
\[
   |t-x|^{-p\theta}\in L^1_{\mathrm{loc}}
\]
whenever
\[
   p\theta<1.
\]
Thus
\[
   \rho\in L^p(I)
   \qquad
   \text{for every }
   1<p<\frac1\theta.
\]
Since $\theta<1$, such a $p>1$ always exists.
\end{proof}

\bibliographystyle{amsplain}
\bibliography{MS_determinant_references}

\providecommand{\bysame}{\leavevmode\hbox to3em{\hrulefill}\thinspace}
\providecommand{\MR}{\relax\ifhmode\unskip\space\fi MR }
\providecommand{\MRhref}[2]{%
  \href{http://www.ams.org/mathscinet-getitem?mr=#1}{#2}
}
\providecommand{\href}[2]{#2}
\begin{thebibliography}{10}

\bibitem{AjankiErdosKruger2019}
Oskari~H. Ajanki, L{\'a}szl{\'o} Erd{\H o}s, and Torben Kr{\"u}ger,
  \emph{Stability of the matrix {Dyson} equation and random matrices with
  correlations}, Probability Theory and Related Fields \textbf{173} (2019),
  no.~1-2, 293--373.

\bibitem{Brown1986}
Lawrence~G. Brown, \emph{Lidskii's theorem in the type {II} case}, Geometric
  Methods in Operator Algebras (Kyoto, 1983) (Huzihiro Araki and Edward~G.
  Effros, eds.), Pitman Research Notes in Mathematics Series, vol. 123, Longman
  Scientific \& Technical, Harlow, 1986, pp.~1--35.

\bibitem{EvansHoeghKrohn1978}
David~E. Evans and Raphael H{\o}egh-Krohn, \emph{Spectral properties of
  positive maps on {$C^*$}-algebras}, Journal of the London Mathematical
  Society \textbf{17} (1978), no.~2, 345--355.

\bibitem{FugledeKadison1952}
Bent Fuglede and Richard~V. Kadison, \emph{Determinant theory in finite
  factors}, Annals of Mathematics \textbf{55} (1952), no.~3, 520--530.

\bibitem{GargGurvitsOliveiraWigderson2020}
Ankit Garg, Leonid Gurvits, Rafael Oliveira, and Avi Wigderson, \emph{Operator
  scaling: theory and applications}, Foundations of Computational Mathematics
  \textbf{20} (2020), no.~2, 223--290.

\bibitem{Gurvits2004}
Leonid Gurvits, \emph{Classical complexity and quantum entanglement}, Journal
  of Computer and System Sciences \textbf{69} (2004), no.~3, 448--484.

\bibitem{HaagerupLarsen2000}
Uffe Haagerup and Flemming Larsen, \emph{Brown's spectral distribution measure
  for {$R$}-diagonal elements in finite von {N}eumann algebras}, J. Funct.
  Anal. \textbf{176} (2000), no.~2, 331--367.

\bibitem{HallHo2025}
Brian~C. Hall and Ching-Wei Ho, \emph{Spectral results for free random
  variables}, 2025, arXiv:2510.03382.

\bibitem{HansenPedersen1982}
Frank Hansen and Gert~Kj{\ae}rg{\aa}rd Pedersen, \emph{Jensen's inequality for
  operators and {L{\"o}wner}'s theorem}, Mathematische Annalen \textbf{258}
  (1982), no.~3, 229--241.

\bibitem{HeltonRashidiFarSpeicher2007}
J.~William Helton, Reza Rashidi~Far, and Roland Speicher, \emph{Operator-valued
  semicircular elements: solving a quadratic matrix equation with positivity
  constraints}, International Mathematics Research Notices \textbf{2007}
  (2007), no.~22, Art. ID rnm086, 15.

\bibitem{KarginNonSingular2026}
Vladislav Kargin, \emph{An algebraic characterization of non-singular matrix
  semicircles}, 2026, arXiv:2604.23089.

\bibitem{KarginSingularities2026}
\bysame, \emph{Singularities of matrix semicircles}, 2026, arXiv:2607.26222.

\bibitem{Kunisky2026}
Dmitriy Kunisky, \emph{Lehner's operator norm formulas, semidefinite
  programming, and spiked matrix models}, 2026, arXiv:2606.14687.

\bibitem{Lehner1999}
Franz Lehner, \emph{Computing norms of free operators with matrix
  coefficients}, American Journal of Mathematics \textbf{121} (1999), no.~3,
  453--486.

\bibitem{MaiSpeicher2024}
Tobias Mai and Roland Speicher, \emph{Fuglede--kadison determinants of
  matrix-valued semicircular elements and capacity estimates}, 2024,
  arXiv:2406.15922.

\bibitem{MaiSpeicherYin2023}
Tobias Mai, Roland Speicher, and Sheng Yin, \emph{The free field: realization
  via unbounded operators and {A}tiyah property}, J. Funct. Anal. \textbf{285}
  (2023), no.~5, 110016, arXiv:1905.08187.

\bibitem{Miyagawa2026}
Akihiro Miyagawa, \emph{The spectra of polynomials in free (semi)circular
  operators}, 2026, arXiv:2603.19528.

\bibitem{ShlyakhtenkoSkoufranis2015}
Dimitri Shlyakhtenko and Paul Skoufranis, \emph{Freely independent random
  variables with non-atomic distributions}, Transactions of the American
  Mathematical Society \textbf{367} (2015), no.~9, 6267--6291.

\bibitem{SpeicherBlog2026}
Roland Speicher, \emph{A story of the determinant of operator-valued
  semicircular elements}, Blog post, \emph{Free Probability Theory}, September
  2026,
  \url{https://rolandspeicher.com/2026/09/09/a-story-of-the-determinant-of-operator-valued-semicircular-elements/},
  posted 9 September 2026.

\bibitem{Walker1950}
Robert~J. Walker, \emph{Algebraic curves}, Princeton Mathematical Series,
  vol.~13, Princeton University Press, Princeton, NJ, 1950.

\bibitem{Yin2026}
Sheng Yin, \emph{A variational formula for affine semicircular determinants},
  Hexagon preprint hexagon:2610.00180v1, 2026,
  \url{https://hexagonmath.org/2610.00180v1}.

\end{thebibliography}

\end{document}